\documentclass[12pt]{amsart}
\usepackage{amssymb}
\newtheorem{theorem}{Theorem}[section]
\newtheorem{lemma}[theorem]{Lemma}
\newtheorem{proposition}[theorem]{Proposition}
\newtheorem{corollary}[theorem]{Corollary}

\theoremstyle{definition}

\newtheorem{example}[theorem]{Example}

\theoremstyle{remark}
\newtheorem{remark}[theorem]{Remark}

\numberwithin{equation}{section}

\newcommand{\QED}{\qed}
\newcommand{\bfind}[1]{\index{#1}{\bf #1}}

\newcommand{\n}{\par\noindent}
\newcommand{\nn}{\par\vskip2pt\noindent}
\newcommand{\sn}{\par\smallskip\noindent}
\newcommand{\mn}{\par\medskip\noindent}

\newcommand{\pars}{\par\smallskip}
\newcommand{\parm}{\par\medskip}
\newcommand{\parb}{\par\bigskip}
\newcommand{\isom}{\simeq}
\newcommand{\fvklit}[1]{[#1]}

\newcommand{\ovl}[1]{\overline{#1}}

\newcommand{\sep}{^{\rm sep}}
\newcommand{\chara}{\mbox{\rm char}\,}

\newcommand{\dist}{\mbox{\rm dist}\,}
\newcommand{\Gal}{\mbox{\rm Gal}\,}
\newcommand{\im}{\,\mbox{\rm im}\,}

\newcommand{\subsetuneq}{\mathrel{\raisebox{.8ex}{\footnotesize%
$\displaystyle\mathop{\subset}_{\not=}$}}}
\newcommand{\cal}{\mathcal}

\newcommand{\N}{\mathbb N}
\newcommand{\Q}{\mathbb Q}

\newcommand{\Z}{\mathbb Z}

\newcommand{\F}{\mathbb F}

\begin{document}
\title[Artin-Schreier defect extensions]{A classification of
Artin-Schreier defect extensions and characterizations of defectless
fields}
\author{Franz-Viktor Kuhlmann}
\address{Mathematical Sciences Group,
University of Saskatchewan,
106 Wiggins Road,
Saskatoon, Saskatchewan, Canada S7N 5E6}
\email{fvk@math.usask.ca}
\thanks{I thank Peter Roquette
for his invaluable help and support, and Olivier Piltant and Bernard
Teissier for inspiring discussions. Parts of this paper were written
while I was a guest of the Equipe G\'eom\'etrie et Dynamique, Institut
de Math\'ematiques de Jussieu, Paris, and of the Laboratoire de
Math\'ematiques at the Universit\'e de Versailles
Saint-Quentin-en-Yvelines. I gratefully acknowledge their hospitality
and support. I also thank the referee for his careful reading and an
abundance of useful remarks and corrections. I was partially
supported by a Canadian NSERC grant and by sabbatical grants from the
University of Saskatchewan.}
\date{16.\ 10.\ 2009}
\subjclass[2000]{Primary 12J20; Secondary 12J10, 12F05}
\begin{abstract}\noindent
{\footnotesize\rm
We classify Artin-Schreier extensions of valued fields with non-trivial
defect according to whether they are connected with purely inseparable
extensions with non-trivial defect, or not. We use this classification
to show that in positive characteristic, a valued field is algebraically
complete if and only if it has no proper immediate algebraic extension
and every finite purely inseparable extension is defectless. This result
is an important tool for the construction of algebraically complete
fields. We also consider extremal fields (= fields
for which the values of the elements in the images of arbitrary
polynomials always assume a maximum).
We characterize inseparably defectless, algebraically maximal and
separable-algebraically maximal fields in terms of extremality,
restricted to certain classes of polynomials. We give a second
characterization of algebraically complete fields, in terms of their
completion. Finally, a variety of examples for Artin-Schreier extensions
of valued fields with non-trivial defect is presented.}
\end{abstract}
\maketitle
%
%
%
%
\section{Introduction}
In this paper, we consider fields $K$ equipped with (Krull) valuations
$v$. The value group of $(K,v)$ will be denoted by $vK$, and its residue
field by $Kv$. The value of an element $a$ is denoted by $va$, and its
residue by $av$. We will frequently drop the valuation $v$ and talk of
$K$ as a valued field if the context is clear. By $(L|K,v)$ we mean an
extension of valued fields where $v$ is a valuation on $L$ and its
subfield $K$ is endowed with the restriction of $v$. The extension
$(L|K,v)$ is called \bfind{immediate} if $(vL:vK)=1$ and $[Lv:Kv]=1$.

Assume that $(L|K,v)$ is a finite extension and the extension of $v$
from $K$ to $L$ is unique. Then the Lemma of Ostrowski tells us that
\begin{equation}                            \label{LoO}
[L:K]\;=\; (vL:vK)\cdot [Lv:Kv]\cdot p^\nu \;\;\;\mbox{ with }\nu\geq 0
\end{equation}
where $p$ is the \bfind{characteristic exponent} of $Kv$, that is,
$p=\chara Kv$ if this is positive, and $p=1$ otherwise. The factor
${\rm d}(L|K):=p^\nu$ is called the \bfind{defect} of the extension
$(L|K,v)$. If $\nu>0$, then we talk of a \bfind{non-trivial} defect and
call $(L|K,v)$ a \bfind{defect extension}. Otherwise, we call $(L|K,v)$
a \bfind{defectless extension}. If $[L:K]=p$ then $(L|K,v)$ is a defect
extension if and only if it is immediate. A possibly infinite algebraic
extension is called defectless if all of its finite subextensions are
defectless; in view of Lemma~\ref{md} in Section~\ref{sectddie}, this
agrees with the definition for finite extensions.

\pars
Every finite extension $L$ of a valued field $(K,v)$
satisfies the \bfind{fundamental inequality} (cf.\ [En], [Z--S]):

\begin{equation}                             \label{fundineq}
n\>\geq\>\sum_{i=1}^{\rm g} {\rm e}_i {\rm f}_i
\end{equation}
where $n=[L:K]$ is the degree
of the extension, $v_1,\ldots,v_{\rm g}$ are the distinct extensions
of $v$ from $K$ to $L$, ${\rm e}_i=(v_iL:vK)$ are the respective
\bfind{ramification indices} and ${\rm f}_i=[Lv_i:Kv]$ are the
respective \bfind{inertia degrees}. If ${\rm g}=1$ for every finite
extension $L|K$ then $(K,v)$ is called \bfind{henselian}. This holds if
and only if $(K,v)$ satisfies \bfind{Hensel's Lemma}, that is, if $f$ is
a polynomial with coefficients in the valuation ring ${\cal O}$ of
$(K,v)$ and there is $b\in {\cal O}$ such that $vf(b)>0$ and $vf'(b)=0$,
then there is $a\in {\cal O}$ such that $f(a)=0$ and $v(b-a)>0$.

Every valued field $(K,v)$ has a minimal separable-algebraic extension
which is henselian; it is unique up to isomorphism over $K$. We call it
the \bfind{henselization of} $(K,v)$ and denote it by $(K,v)^h$. It is
an immediate extension of $(K,v)$.

We call a (not necessarily henselian) valued field $(K,v)$ a
\bfind{defectless field}, \bfind{separably defectless field} or
\bfind{inseparably defectless field} if equality holds in the
fundamental inequality (\ref{fundineq}) for every finite, finite
separable or finite purely inseparable extension $L$ of $K$. One can
trace this back to the case of unique extensions of the valuation,
respectively; for the proof of the following theorem, see [Ku9] (a
partial proof was already given in [En]):
\begin{theorem}                             \label{dhd}
A valued field $(K,v)$ is a defectless field if and only if its
henselization $(K,v)^h$ is (that is, if and only if every finite
extension of $(K,v)^h$ is a defectless extension). The same holds for
``separably defectless field'' and ``inseparably defectless field'' in
the place of ``defectless field''.
\end{theorem}

A valued field is called \bfind{algebraically complete} if it is
henselian and defectless.

\par\smallskip
For various reasons (e.g., local uniformization in positive
characteristic [Kn-Ku1], [Kn-Ku2], model theory of valued fields [Ku6])
it is necessary to study the structure of defect extensions. A
ramification theoretic method that was used frequently by S.~Abhyankar
and that is also employed in [Ku4] is to consider the part of an
extension $(L|K,v)$ that ``lies above'' its ramification field. We can
reformulate this in the following way. We let $K^r$ denote the
\bfind{absolute ramification field} of $K$, i.e., the ramification field
of the extension $K\sep|K$ with respect to a fixed extension of $v$ to
the separable-algebraic closure $K\sep$ of $K$. Then we consider the
extension $L.K^r|K^r$. This extension has the same defect as $L|K$ (cf.\
Proposition~\ref{dlta} below). On the other hand, the Galois group of
$K\sep|K^r$ is a pro-$p$-group (cf.\ [En] or [N]). Consequently,
$L.K^r|K^r$ is a tower of normal extensions $L_1|L_2$ of degree $p$
(cf.\ Lemma~\ref{tow} in Section~\ref{sectddie}). If our fields have
characteristic $p$ and if $L_1|L_2$ is separable, then it is an
\bfind{Artin-Schreier extension}, that is, it is generated by a root
$\vartheta$ of a polynomial of the form $X^p-X-a$ with $a\in L_2$ (see,
e.g., [L]); in this case, $\vartheta$ is called an \bfind{Artin-Schreier
generator} of $L_1|L_2$. Such extensions are always normal and hence
Galois since the other roots of $X^p-X-a$ are $\vartheta+1,\ldots,
\vartheta+p-1$. This follows from the fact that $0,1,\ldots,p-1$ are all
roots of the Artin-Schreier polynomial $\wp(X) =X^p-X$ and that this
polynomial is additive. A polynomial $f\in K[X]$ is called
\bfind{additive} if $f(b+c)=f(b)+f(c)$ for all $b,c$ in every
extension field of $K$ (cf.\ [L], [Ku3]).

Also in the mixed characteristic case where $\chara L_1=0$ and $\chara
L_1v=p$, we will call $L_1|L_2$ an Artin-Schreier extension with
Artin-Schreier generator $\vartheta$ if $[L_1:L_2]=p$,
$L_1=L_2(\vartheta)$ and $\vartheta^p-\vartheta\in L_2$.

Because of the representative role of Artin-Schreier extensions that we
just pointed out, it is interesting to know more about their structure,
in particular when they have non-trivial defect. In this paper, we will
classify Artin-Schreier defect extensions according to the question
whether they are in some sense similar to immediate purely inseparable
extensions. Then we study the relation between the two different types
of extensions in our classification. In Section~\ref{sectexamp} we will
give several examples of Artin-Schreier defect extensions of both types.

The defect is a bad phenomenon as it destroys the tight connection
between valued fields and their invariants, value group and residue
field. Therefore, it is desirable to work with defectless fields. As the
notion of ``defectless field'' plays an important role in several
applications,
%
%
it is helpful to have equivalent characterizations. For instance,
when it comes to constructing defectless fields, one would like to have
criteria that could (more or less) easily be checked. Our results on
Artin-Schreier defect extensions will enable us to break down the
property ``defectless field'' into weaker maximality properties of
valued fields.

A valued field $(K,v)$ is called \bfind{algebraically maximal} if it has
no proper immediate algebraic extensions, and
\bfind{separable-algebraically maximal} if it has no proper immediate
separable-algebraic extensions. Note that a separable-algebraically
maximal valued field is henselian, because the henselization is an
immediate separable-algebraic extension. In Section~\ref{sectth1}, we
will prove the following useful characterization of the property
``defectless field'':
\begin{theorem}                             \label{si=d}
A valued field of positive characteristic is henselian and defectless
if and only if it is separable-algebraically maximal and inseparably
defectless.
\end{theorem}
\n
This characterization has been applied in [Ku1] to construct a
valued field extension of the field $\F_p((t))$ of formal Laurent series
over the field with $p$ elements which is henselian defectless with
value group a $\Z$-group and residue field $\F_p$ but does not satisfy a
certain elementary sentence (involving additive polynomials) that holds
in $\F_p((t))$. This example shows that the axiom system ``henselian
defectless valued field of characteristic $p$ with value group a
$\Z$-group and residue field $\F_p$'' is not complete. Whenever one
wants to construct a henselian defectless field, the problem is to get
it to be defectless. It is easy to make it henselian (just go to the
henselization) or even algebraically maximal (just go to a maximal
immediate algebraic extension). But the latter does not imply that the
field is defectless, as an example given by F.~Delon [D1] shows (see
also [Ku5]). However, in the case of finite $p$-degree, Delon also gave
a handy characterization of inseparably defectless valued fields, see
Theorem~\ref{charinsdl}. (Recall that $d$ is called the
\bfind{$p$-degree}, or \bfind{Ershov invariant}, or \bfind{degree of
imperfection}, of $K$ if $[K:K^p]=p^d$.) Together with the above
theorem, this provides a handy characterization of henselian defectless
fields of characteristic $p$ in the case of finite $p$-degree.

\pars
Take a valued field $(K,v)$ with valuation ring ${\cal O}$. If $f$ is a
polynomial in $n$ variables with coefficients in $K$, then we will say
that $(K,v)$ is \bfind{$K$-extremal with respect to $f$} if
the set
\begin{equation}                            \label{defextr}
v\im_K (f)\;:=\;\{vf(a_1,\ldots,a_n)\mid a_1,\ldots,a_n\in
K\}\>\subseteq vK\cup\{\infty\}
\end{equation}
has a maximum, and we we will say that $(K,v)$ is \bfind{$\cal O$-extremal
with respect to $f$} if the set
\begin{equation}                            \label{defextrO}
v\im_{\cal O} (f)\;:=\;\{vf(a_1,\ldots,a_n)\mid a_1,\ldots,a_n\in
{\cal O}\}\>\subseteq vK\cup\{\infty\}
\end{equation}
has a maximum. The former means that
\[\exists Y_1,\ldots,Y_n \forall X_1,\ldots,X_n:\;
vf(X_1,\ldots,X_n)\leq vf(Y_1,\ldots,Y_n)\]
holds in $(K,v)$. For the latter, one has to build into the sentence the
condition that the $X_i$ and $Y_j$ only run over elements of ${\cal O}$.
It follows that being $K$-extremal or ${\cal O}$-extremal with respect
to $f$ is an elementary property in the language of valued fields with
parameters from $K$. Note that in the first case the maximum is $\infty$
if and only if $f$ admits a zero in $K^n$; in the second case, this
zero has to lie in ${\cal O}^n$. A valued field $(K,v)$ is called
\bfind{extremal} if for all $n\in\N$, it is ${\cal O}$-extremal with
respect to every polynomial $f$ in $n$ variables with coefficients in
$K$. This property can be expressed by a countable scheme of elementary
sentences (quantifying over the coefficients of all possible polynomials
of degree at most $n$ in at most $n$ variables). Hence, it is elementary
in the language of valued fields.

If we would have chosen $K$-extremality for the definition of ``extremal
valued field'' (as Yu.\ Ershov in [Er2]), then we would have obtained
precisely the class of algebraically closed valued fields. Using ${\cal
O}$-extremality instead yields a much more interesting class of valued
fields. See [A--Ku--Pop] for details.

The properties ``algebraically maximal'', ``separable-algebraically
\linebreak
maximal'' and ``inseparably defectless'' are each equivalent to
$K$- or ${\cal O}$-extrema\-lity restricted to certain (elementarily
definable) classes of polynomials. A polynomial is called a
\bfind{$p$-polynomial} if it is of the form ${\cal A}(X)+c$ where
${\cal A}(X)$ is an additive polynomial and $c$ is a constant.
We say that a basis $b_1,\ldots,b_n$ of a valued field extension
$(L|K,v)$ is a \bfind{valuation basis} if for all choices of
$c_1,\ldots,c_n\in K$,
\[v\sum_{i=1}^{n}c_ib_i\;=\;\min_i vc_ib_i\>.\]

In Section~\ref{sectidf}, we prove:
\begin{theorem}                             \label{extrid}
A valued field $K$ of positive characteristic is inseparably defectless
if and only if it is $K$-extremal with respect to every $p$-polynomial
of the form
\begin{equation}                            \label{b-bixp}
b-\sum_{i=1}^{n} b_iX_i^p
\end{equation}
with $n\in\N,\>b,b_1,\ldots,b_n\in K$ such that $b_1,\ldots,b_n$ form a
basis of a finite extension of $K^p$ (inside of $K$). If the value group
$vK$ of $K$ is divisible or a $\Z$-group, then $K$ is inseparably
defectless if and only if it is ${\cal O}$-extremal with respect to
every $p$-polynomial (\ref{b-bixp}) with $n\in\N,\>b,b_1,\ldots,b_n\in
{\cal O}$ such that $b_1,\ldots,b_n$ form a valuation basis of a finite
defectless extension of $K^p$ and $vb_1,\ldots,vb_n$ are smaller than
every positive element of $vK$.
\end{theorem}

In Section~\ref{sectamf}, we prove:
\begin{theorem}                             \label{extram}
A valued field $K$ is algebraically maximal if and only if it is
${\cal O}$-extremal with respect to every polynomial in one variable
with coefficients in $K$.
\end{theorem}

\begin{theorem}                             \label{extramh}
A henselian valued field $K$ of positive characteristic is algebraically
maximal if and only if it is ${\cal O}$-extremal with respect to every
$p$-polynomial in one variable with coefficients in $K$.
\end{theorem}

In Section~\ref{sectsamf}, we prove:
\begin{theorem}                             \label{extrsam}
A valued field $K$ is separable-algebraically maximal if and only if it
is ${\cal O}$-extremal with respect to every separable polynomial in one
variable with coefficients in $K$.
\end{theorem}

\begin{theorem}                             \label{extrsamh}
A henselian valued field $K$ of positive characteristic is
separable-algebraically maximal if and only if it is ${\cal O}$-extremal
with respect to every separable $p$-polynomial in one variable with
coefficients in $K$.
\end{theorem}

\begin{theorem}                             \label{KO}
In Theorems~\ref{extram}, \ref{extramh}, \ref{extrsam}
and~\ref{extrsamh}, ``${\cal O}$-extremal'' can be replaced by
``$K$-extremal''.
\end{theorem}
\sn
Theorems~\ref{extram} and~\ref{extrsam} in the ``$K$-extremal'' version
were presented by Delon in [D1], but the proofs had gaps in both
directions.

\pars
In [A--Ku--Pop], we prove:
\begin{theorem}                             \label{extr=>hd}
Every extremal field is henselian and defectless. Every finite extension
of an extremal field is again extremal.
\end{theorem}
\n
These results were proved by Yu.~Ershov in [Er2] for ``$K$-extremal'' in
the place of ``${\cal O}$-extremal''. But in this case they are trivial
consequences of the fact that every $K$-extremal valued field is
algebraically closed (cf.\ [A--Ku--Pop]).

\pars
We also obtain that the properties ``algebraically maximal'',
``separ\-able-algebraically maximal'' and ``inseparably defectless''
are elementary in the language of valued fields, see
Corollary~\ref{extr1el}, Corollary~\ref{extrsamel} and
Corollary~\ref{idelem}. By Theorem~\ref{si=d}, this fact provides
an easy proof of the following result, which was proved by Ershov [Er1],
and independently by Delon [D1], by different methods:

\begin{theorem}
The property ``henselian and defectless valued field of characteristic
$p>0$'' is elementary in the language of valued fields.
\end{theorem}

\pars
In Section~\ref{sectac}, we will give another characterization of
henselian defectless fields, in terms of their completion
(Theorem~\ref{:}). We will also show that a henselian field of positive
characteristic is separably defectless if and only if its completion is
defectless (Theorem~\ref{csdd}).

A field of positive characteristic is called \bfind{Artin-Schreier
closed} if it admits no non-trivial Artin-Schreier extensions. In
Section~\ref{sectfwdASe} we will prove:
\begin{theorem}                             \label{AScphic}
Every Artin-Schreier closed non-trivially valued field lies dense in its
perfect hull, and its completion is perfect. In particular, every
separable-algebraically closed non-trivially valued field lies dense in
its algebraic closure.
\end{theorem}
\n
The second part of this theorem is well known (cf.\ [W], Theorem~30.28).

\parm
Several of the results of this paper, and in particular
Theorem~\ref{si=d}, have been inspired by the work presented in
Francoise Delon's thesis [D1].

%
%
\section{Preliminaries}
For the basic facts of valuation theory, we refer the reader to [En],
[Ri], [W] and [Z--S]. For ramification theory, we recommend [En] and
[N]. In parts of this paper we will assume some familiarity with the
theory of pseudo Cauchy sequences as presented in the first half of
[Ka]. Note that our ``pseudo Cauchy sequence'' is what Kaplansky calls
``pseudo convergent set''.

The algebraic closure of a field $K$ will be denoted by $\tilde{K}$.
Note that if $v$ is a valuation on $K$, then any possible extension of
$v$ to $\tilde{K}$ has residue field $\widetilde{Kv}$ and value group
$\widetilde{vK}$, the divisible hull of $vK$ (isomorphic to $\Q\otimes
vK$).

%
%
\subsection{Defectless extensions, defect and immediate
extensions}                                 \label{sectddie}
The defect is multiplicative in the following sense. Let $L|K$ and $M|L$
be finite extensions. Assume that the extension of $v$ from $K$ to $M$
is unique. Then the defect satisfies the following product formula
\begin{equation}         \label{pf}
\mbox{\rm d}(M:K) = \mbox{\rm d}(M:L)\cdot\mbox{\rm d}(L:K)
\end{equation}
which is a consequence of the multiplicativity of the degree of field
extensions and of ramification index and inertia degree.
This formula implies:

\begin{lemma}                                         \label{md}
$M | K$ is defectless if and only if $M | L$ and $L | K$ are defectless.
\end{lemma}

Together with Theorem~\ref{dhd}, this lemma yields:

\begin{corollary}      \label{ed}
If $(K,v)$ is a defectless field and $L$ is a finite extension of $K$,
then $L$ is also a defectless field with respect to every extension of
$v$ from $K$ to $L$. Conversely, if there exists a finite extension $L$
of $K$ for which equality holds in the fundamental
inequality (\ref{fundineq}) and such that $L$ is a defectless field with
respect to every extension of $v$ from $K$ to $L$, then $(K,v)$ is a
defectless field. The same holds for ``separably defectless'' in the
place of ``defectless'' if $L|K$ is separable, and for ``inseparably
defectless'' if $L|K$ is purely inseparable.
\end{corollary}

Recall that an infinite algebraic extension $(L|K,v)$ with unique
extension of the valuation is called defectless if every finite
subextension is defectless. This definition is compatible with the
definition of ``defectless'' for finite extensions as given in the
introduction, because by Lemma~\ref{md} every subextension of a finite
defectless extension is again defectless. We have:

\begin{lemma}                               \label{transdl}
Let $L|K$ and $L'|L$ be (not necessarily finite) extensions of valued
fields such that the extension of $v$ from $K$ to $L'$ is unique. If
both $L|K$ and $L'|L$ are defectless, then so is $L'|K$.
\end{lemma}
\begin{proof}
Let $F|K$ be any finite subextension of $L'|K$. Since $L'|L$ is
defectless, so is its finite subextension $F.L|L$.
Hence, $[F.L:L]=(v(F.L):vL)[(F.L)v:Lv]$. Pick a set $\alpha_1,\ldots,
\alpha_k$ of generators of $vF.L$ over $vL$, and a basis $\zeta_1,
\ldots,\zeta_{\ell}$ of $(F.L)v|Lv$. Choose a finite subextension
$L_0|K$ of $L|K$ such that
\[[F.L_0:L_0]=[F.L:L]\,,\; \alpha_1,\ldots,\alpha_k\in v(F.L_0)\,, \;
\zeta_1,\ldots,\zeta_{\ell}\in (F.L_0)v\;.\]
Then $(v(F.L_0):vL_0)\geq (v(F.L):vL)$, $[(F.L_0)v:L_0v]\geq
[(F.L)v:Lv]$, and
\begin{eqnarray*}
[F.L:L] & = & [F.L_0:L_0]\>\geq\>(v(F.L_0):vL_0)[(F.L_0)v:L_0v]\\
 & \geq & (v(F.L):vL)[(F.L)v:Lv]\>=\>[F.L:L]\;,
\end{eqnarray*}
where the first inequality follows from (\ref{LoO}).
Hence, equality must hold everywhere, and we obtain that $F.L_0|L_0$
is defectless. Since $L|K$ is assumed to be defectless, also
$L_0|K$ is defectless. By Lemma~\ref{md} it follows that $F.L_0|K$
is defectless. Again by Lemma~\ref{md}, also the subextension $F|K$
is defectless. This proves that $L'|K$ is defectless.
\end{proof}

Let $(K,v)$ be any valued field. A valuation $w$ on $K$ is a
\bfind{coarsening} of $v$ if its valuation ring ${\cal O}_w$ contains
the valuation ring ${\cal O}_v$ of $v$. Note that we do not exclude the
case of $w=v$. If $H$ is a convex subgroup of $vK$, then it gives rise
to a coarsening $w$ through the definition ${\cal O}_w:= \{x\in K\mid
\exists\alpha\in H:\>\alpha \leq vx\}$. Then $v$ induces a valuation
$\ovl{w}$ on $Kw$ through the definition ${\cal O}_{\ovl{w}}:=\{xw\mid
x\in {\cal O}_v\}$, and there are canonical isomorphisms $wK\isom vK/H$
and $\ovl{w}(Kw)\isom H$.
If $(K,w)$ is any valued field and if $w'$ is any valuation on the
residue field $Kw$, then $w\circ w'$, called the \bfind{composition of
$w$ and $w'$}, will denote the valuation whose valuation ring is the
subring of the valuation ring of $w$ consisting of all elements whose
$w$-residue lies in the valuation ring of $w'$. (Note that we identify
equivalent valuations.) In our above situation, $v$ is the composition
of $w$ and $\ovl{w}$. While $w\circ w'$ does actually not mean the
composition of $w$ and $w'$ as mappings, this notation is used because
in fact, up to equivalence the place associated with $w\circ w'$ is
indeed the composition of the places associated with $w$ and $w'$.

\begin{lemma}               \label{defc}
Take a henselian field $(K,v)$, a finite extension $(L|K,v)$ and a
coarsening $w$ of $v$ on $L$. Then also $(K,w)$ is henselian. If
$(L|K,v)$ is defectless, then also $(L|K,w)$ is defectless.
\end{lemma}
\begin{proof}
If there are two distinct extensions $w_1$ and $w_2$ of $w$ from $K$
to $\tilde{K}$, and if we take any extension of
$\ovl{w}$ to the algebraic closure $\widetilde{Kw}=\tilde{K}w_1=
\tilde{K}w_2$ of $Kw$, then the compositions $w_1\circ\ovl{w}$ and
$w_2\circ\ovl{w}$ will be distinct extensions of $v$ to $\tilde{K}$.
This shows that $(K,v)$ cannot be henselian if $(K,w)$ isn't.

Now assume that $(L|K,v)$ is defectless, that is, $[L:K]=(vL:vK)\cdot
[Lv:Kv]$. We have that $(vL:vK)=(wL:wK)(\ovl{w}(Lw):\ovl{w}(Kw))$,
$(Lw)\ovl{w}=Lv$ and $(Kw)\ovl{w}=Kv$. Therefore,
\begin{eqnarray*}
[L:K] & \geq & (wL:wK)[Lw:Kw]\\
& \geq & (wL:wK)(\ovl{w}(Lw):\ovl{w}(Kw))[(Lw)\ovl{w}:(Kw)\ovl{w}]\\
& = & (vL:vK)[Lv:Kv]\;=\;[L:K]\;.
\end{eqnarray*}
This shows that equality holds everywhere, which proves that
$(L|K,w)$ is defectless.
\end{proof}


In the next lemma, the relation between immediate and defectless
extensions is studied.
\begin{lemma}                        \label{ivd}
Let $K$ be a valued field and $F|K$ an arbitrary immediate extension. If
$L|K$ is a finite defectless extension admitting a unique extension of
the valuation, then the same holds for $F.L|F$, and $F.L|L$ is
immediate. Moreover,
\[[F.L:F] = [L:K]\;,\]
i.e., $F$ is linearly disjoint from $L$ over $K$.
\end{lemma}
\begin{proof}
$v(F.L)$ contains $vL$ and $(F.L)v$ contains $Lv$. On the other
hand, we have $vF=vK$ and $Fv = Kv$ by hypothesis. Therefore,
\begin{eqnarray*}
[F.L:F] & \geq & (v(F.L):vF)\cdot[(F.L)v:Fv]\\
 & \geq & (vL:vK)\cdot [Lv:Kv] \>=\> [L:K] \geq [F.L:F]\;,
\end{eqnarray*}
hence equality holds everywhere. This shows that $[F.L:F] = [L:K]$ and
that $F.L|F$ is defectless. Furthermore it follows that $v(F.L) = vL$
and $(F.L)v=Lv$, i.e., $F.L|L$ is immediate.
\end{proof}

As an immediate consequence we get:
\begin{corollary}                          \label{lid}
If $K$ is an inseparably defectless field then every immediate extension
is separable. If $K$ is a henselian defectless field then every
immediate extension is regular.
\end{corollary}

Let $(K,v)$ be a henselian field and $p$ the characteristic exponent of
its residue field $Kv$. An algebraic extension $(L|K,v)$ is called
a \bfind{tame extension} if for every finite subextension $(L_0|K,v)$,
the following conditions hold:
\sn
1) \ $p$ is prime to $(vL_0:vK)$,\n
2) \ $L_0v|Kv$ is separable,\n
3) \ $(L_0|K,v)$ is defectless.
\sn
On the other hand, an algebraic extension $(L|K,v)$ is called a
\bfind{purely wild extension} if
\sn
1) \ $vL/vK$ is a $p$-group,\n
2) \ $Lv|Kv$ is purely inseparable.
\sn
If $p=1$, that is, $\chara Kv=0$, then every algebraic extension of a
henselian field is tame, and only the trivial extension is purely wild.

By Proposition 4.1 of [Ku--P--Ro], the absolute ramification field $K^r$
of $K$ is the unique maximal tame algebraic extension of $K$; it is a
normal separable extension of $K$. By Lemma 4.2 of [Ku--P--Ro], an
algebraic extension $L|K$ is purely wild if and only if it is linearly
disjoint from $K^r|K$. From these facts it follows that $K^rv$ is the
separable-algebraic closure of $Kv$ and $vK^r$ is the \bfind{$p$-prime
divisible hull} of $vK$:
\begin{equation}                            \label{vgrfK^R}
vK^r\>=\>\bigcup_{n\in\N\setminus p\N} \frac{1}{n}\,\Z
\;\;\;\mbox{ and }\;\;\; K^rv\>=\> (Kv)\sep\;.
\end{equation}

We will now consider the behaviour of the defect when a finite extension
$L|K$ of a henselian field $K$ is shifted up through a tame extension
$N|K$. We need the following information on $K^r$:

\begin{lemma}                    \label{rf}
Let $(K,v)$ be an arbitrary valued field and $p$ the characteristic
exponent of $Kv$. Take an algebraic extension $L|K$. Then $L^r=L.K^r$;
hence if $L\subset K^r$, then $L^r=K^r$. The separable--algebraic
closure $K\sep$ is a $p$--extension of $K^r$.
\end{lemma}
\begin{proof}
For separable extensions, the first assertion follows from [En], page
166, (20.15) b) (where we put $N= K\sep$ since we define $K^r$ to be
the ramification field of the separable extension $K\sep|K\,$). Since
every algebraic extension can be viewed as a purely inseparable
extension of a separable extension, it remains to show the first
assertion for a purely inseparable extension $L|K$. Here, it follows
from the fact that Gal$(K) \cong$ Gal$(L)$ and that by this isomorphism,
the Galois group of an intermediate field $K'$ of $K\sep|K$ is
isomorphic to the Galois group of the intermediate field $L.K'$ of
$L\sep|L$. The second assertion follows from [En],
p.\ 167, Theorem (20.18).
%
%
\end{proof}

The next proposition shows the invariance of the defect under lifting up
through tame extensions.

\begin{proposition}                 \label{dlta}
Let $K$ be a henselian field and $N$ an arbitrary tame algebraic
extension of $K$. If $L|K$ is a finite extension, then
\[\mbox{\rm d}(L|K)\>=\> \mbox{\rm d}(L.N | N)\;.\]
In particular, $L|K$ is defectless if and only if $L.N | N$ is
defectless. This implies: $K$ is a defectless field if and only if $N$
is a defectless field, and the same holds for ``separably defectless''
and ``inseparably defectless'' in the place of ``defectless''.
\end{proposition}
\begin{proof}
Since $N^r=K^r$, it suffices to prove our lemma for the case of $N=K^r$,
because then we obtain
\[\mbox{\rm d}(L|K)=\mbox{\rm d}(L.K^r|K^r)= \mbox{\rm d}(L.N^r|N^r)
=\mbox{\rm d}((L.N).N^r|N^r)=\mbox{\rm d}(L.N|N)\]
for general $N$.

We put $L_0 := L\cap K^r$. We have $L.K^r=L^r$ and $L_0^r=K^r$ by
Lemma~\ref{rf}.
%
%
Since $K^r|K$ is normal, $L$ is linearly disjoint from $K^r=L_0^r$ over
$L_0\,$, and $L|L_0$ is thus a purely wild extension.
%

As a finite subextension of the tame extension $K^r|K$, $L_0|K$ is
defectless. Hence by the multiplicativity of the defect (\ref{pf}),
\begin{equation}                            \label{d=d}
\mbox{\rm d}(L|K) = \mbox{\rm d}(L|L_0)\;.
\end{equation}

It remains to show $\mbox{\rm d}(L|L_0)=\mbox{\rm d}(L.K^r|K^r)$. Since
$L|L_0$ is linearly disjoint from $K^r|L_0\,$, we have
\begin{equation}                                  \label{gr0}
[L^r:K^r]\>=\> [L.K^r:K^r]\>=\> [L:L_0]\;.
\end{equation}
Since $L|L_0$ is purely wild, $vL/vL_0$ is a $p$-group and
$Lv|L_0v$ is purely inseparable. On the other hand,
\pars
$vL^r$ is the $p$-prime divisible hull of $vL$ and $L^rv=(Lv)\sep$,
\par
$vL_0^r$ is the $p$-prime divisible hull of $vL_0$ and
$L_0^rv=(L_0v)\sep$.
\sn
It follows that
\begin{equation}                            \label{vLvL0}
(vL^r:vL_0^r)\>=\>(vL:vL_0)\;\;\;\mbox{ and }\;\;\; [L^rv:L_0^rv]\>=\>
[Lv:L_0v]\;.
\end{equation}
From (\ref{d=d}), (\ref{gr0}) and (\ref{vLvL0}), keeping in mind that
$L.K^r=L^r$ and $L_0^r=K^r$, we deduce
\begin{eqnarray*}
\mbox{\rm d}(L.K^r|K^r) & = & \mbox{\rm d}(L^r|L_0^r)\;=\;
\frac{[L^r:L_0^r]}{(vL^r:vL_0^r)[L^rv:L_0^rv]}\\
& = & \frac{[L:L_0]}{(vL:vL_0)\cdot [Lv:L_0v]}
\;=\; \mbox{\rm d}(L | L_0)\;=\;\mbox{\rm d}(L|K)\;.
\end{eqnarray*}

It now remains to show the second assertion of our proposition. Assume
that $N$ is a defectless field and let $L|K$ be an arbitrary finite
extension. Then by hypothesis, $L.N | N$ is defectless; hence by what we
have shown, $L|K$ is defectless. Since $L|K$ was arbitrary, $K$ is shown
to be a defectless field. Note that $L.N|N$ is separable if $L|K$ is
separable, and $L.N|N$ is purely inseparable if $L|K$ is.

Conversely, assume that $K$ is a defectless field. Since any finite
extension $N' | N$ is contained in an extension $L.N | N$ where $L|K$ is
a finite and, by hypothesis, defectless extension, we see that by what
we have shown, $L.N | N$ and by virtue of Lemma~\ref{md} also its
subextension $N' | N$ are defectless. Note that $L|K$ can be chosen to
be separable if $N'|N$ is separable, and to be purely inseparable if
$N'|N$ is purely inseparable. This completes the proof of our lemma.
\end{proof}

%

\begin{lemma}                    \label{tow}
For every finite extension $L|K$, the extension $L.K^r|K^r$ is a tower
of normal extensions of degree $p$.
For every finite extension $L|K$, there is already a finite
tame extension $N$ of $K^h$ such that $L.N|N$ is such a tower.
\end{lemma}
\begin{proof}
We know from Lemma~\ref{rf} that $K\sep|K^r$ is a $p$--extension, that
is, $\Gal K\sep|K^r$ is a pro-$p$-group. For pro-$p$-groups $G$, the
following is well known: for every open subgroup $H\subset G$ there
exists a chain of open subgroups $H=H_0\subset H_1\subset \ldots\subset
H_n = G$ such that $H_{i-1}\lhd H_{i}$ and $(H_{i}:H_{i-1})=p$ for
$i=1,\ldots,n$. Hence by Galois correspondence, every finite separable
extension of $K^r$ is a tower of Galois extensions of degree $p$. Since
every finite purely inseparable extension is a tower of purely
inseparable extensions of degree $p$, this proves our first assertion.

We take $N$ to be generated over $K^h$ by all the finitely many elements
of $K^r$ that are needed to define the extensions in the tower
$L.K^r|K^r$.
\end{proof}

\begin{corollary}
A valued field $(K,v)$ is henselian and defectless if and only if all of
its finite defectless extensions are algebraically maximal.
\end{corollary}
\begin{proof}
Suppose that $(K,v)$ is henselian and defectless and $(L,v)$ is a finite
extension. Then $(L,v)$ is henselian, and by Corollary~\ref{ed} it is
also defectless.

Suppose now that all finite extensions of $(K,v)$ are algebraically
maximal. Then $(K,v)$ itself is algebraically maximal and hence
henselian. Take any finite extension $(L|K,v)$; we wish to show that it
is defectless. Take a finite tame extension $N|K$ as in the preceding
lemma. Let $L_1|L_2$ be any extension of degree $p$ in the tower $L.N|N$
such that $(L_2|N,v)$ is defectless. Since the finite tame extension
$(N|K,v)$ is defectless, we know by Lemma~\ref{md} that the finite
extension $(L_2|K,v)$ is defectless. So by our hypothesis, $(L_2,v)$ is
algebraically maximal. Thus, $(L_1|L_2,v)$ is not immediate and hence it
is defectless. By induction over the extensions in the tower, together
with repeated applications of Lemma~\ref{md}, this shows that
$(L.N|N,v)$ is defectless. From Proposition~\ref{dlta} it now follows
that $(L|K,v)$ is defectless.
\end{proof}

%
%
\subsection{Immediate extensions and pseudo Cauchy sequences}
%
\begin{lemma}                               \label{apCs}
Take an algebraic extension $(K(a)|K,v)$ and let $f\in K[X]$ be the
minimal polynomial of $a$ over $K$. Suppose that $(c_\nu)_{\nu<\lambda}$
is a pseudo Cauchy sequence in $K$ without limit in $K$, having
$a$ as a limit in $L$. Then $(c_\nu)_{\nu<\lambda}$ is of algebraic
type, and for some $\mu<\lambda$, the values
$(vf(c_\nu))_{\mu<\nu<\lambda}$ are strictly increasing. If in addition,
the extension of $v$ from $K$ to $K(a)$ is unique, then the sequence of
values is cofinal in $v\im_K (f)$, and in particular, $v\im_K (f)$ has
no maximal element.
\end{lemma}
\begin{proof}
Write $f(X)=\prod_{i=1}^n (X-a_i)$ with $a=a_1$ and $a_i\in \tilde{K}$.
Since $a$ is a  limit of $(c_\nu)_{\nu<\lambda}$ we have that
$v(a-c_\nu)=v(c_{\nu+1}-c_\nu)$ is strictly increasing with $\nu$. The
same holds for $a_i$ in the place of $a$ if $a_i$ is also a limit
of $(c_\nu)_{\nu<\lambda}\,$. If it is not, then by Lemma~3 of [Ka]
there is some $\mu_i<\lambda$ such that $v(a-a_i)\leq v(a-c_{\mu_i})$.
For $\mu_i<\nu<\lambda$ we have that $v(a-c_{\mu_i})<v(a-c_\nu)$, which
by the ultrametric triangle law yields $v(a_i-c_\nu)=\min\{v(a-a_i),
v(a-c_\nu)\}=v(a-a_i)$, which does not depend on $\nu$. So if we take
$\mu$ to be the maximum of all these $\mu_i\,$, then for
$\mu<\nu<\lambda$, the value
\[vf(c_\nu)\;=\;v\prod_{i=1}^n (c_\nu-a_i)\;=\;
\sum_{i=1}^{n}v(c_\nu-a_i)\]
is strictly increasing with $\nu$. Consequently, $(c_\nu)_{\nu<\lambda}$
is of algebraic type.

Now assume in addition that the extension of $v$
from $K$ to $K(a)$ is unique. Thus for all $\sigma\in\mbox{\rm Aut}
(\tilde{K}|K)$, the valuations $v$ and $v\circ\sigma$ agree on $K(a)$.
Choosing $\sigma$ such that $\sigma a=a_i\,$, we find that
\[v(a_i-c)\;=\;v(\sigma a-c)\;=\;v\sigma (a-c)\;=\;v(a-c)\]
for all $c\in K$. So we obtain that
\[vf(c)\;=\;\sum_{i=1}^{n}v(c-a_i)\;=\;nv(a-c)\;.\]
Suppose that there is some $c\in K$ such that
\[nv(a-c)\;=\; vf(c)\;>\;vf(c_\nu)\;=\;nv(a-c_\nu)\]
for all $\nu$. This implies that $v(a-c)>v(a-c_\nu)$ which by Lemma~3 of
[Ka] means that $c\in K$ is a limit of $(c_\nu)_{\nu<\lambda}$,
contradicting our hypothesis. This proves that the sequence
$(vf(c_\nu))_{\mu<\nu<\lambda}$ is cofinal in $v\im_K (f)$. Since it has
no last element, it follows that $v\im_K (f)$ has no maximal element.
\end{proof}

\begin{corollary}                           \label{expCsalg}
If $K$ admits a proper immediate algebraic extension, then there is a
pseudo Cauchy sequence of algebraic type in $K$ without a limit in $K$.
\end{corollary}
\begin{proof}
Suppose that $K$ admits a proper immediate algebraic extension $L|K$,
and pick $a\in L\setminus K$. Then by Theorem~1 of [Ka], there is a
pseudo Cauchy sequence in $K$ without a limit in $K$, but having $a$ as
a limit. By the foregoing lemma, this pseudo Cauchy sequence is
of algebraic type.
\end{proof}

%
%
\subsection{Cuts and distances}             \label{sectcad}
Take any totally ordered set $(S,<)$. A \bfind{cut} $\Lambda$ in $S$
is a pair of sets $(\Lambda^{\rm L},\Lambda^{\rm R})$, where:
\sn
a)\ $\Lambda^{\rm L}$ is an \bfind{initial segment} of $S$, i.e., if
$\alpha\in\Lambda^{\rm L}$ and $\beta<\alpha$, then $\beta\in
\Lambda^{\rm L}$,
\n
b)\ $\Lambda^{\rm L}\cup\Lambda^{\rm R}=S$ and $\Lambda^{\rm L}\cap
\Lambda^{\rm R}=\emptyset$ (or equivalently, $\Lambda^{\rm R}=
S\setminus \Lambda^{\rm L}$).
\sn
Note that then, $\Lambda^{\rm R}$ is a \bfind{final segment} of $S$,
i.e., if $\alpha\in\Lambda^{\rm R}$ and $\beta>\alpha$, then
$\beta\in\Lambda^{\rm R}$.

\pars
If $\Lambda_1$ and $\Lambda_2$ are cuts in $S$, then we will write
$\Lambda_1<\Lambda_2$ if $\Lambda_1^{\rm L}\subsetuneq
\Lambda_2^{\rm L}$, and $\Lambda_1=\Lambda_2$ if $\Lambda_1^{\rm L}
=\Lambda_2^{\rm L}$. But we also want to compare two cuts $\Lambda_1$
and $\Lambda_2$ if $\Lambda_2$ is a cut in a totally ordered set $(T,<)$
and $\Lambda_1$ is a cut in some subset $S$ of $T$, endowed with the
restriction of $<$. Here we have at least two canonical ways of
comparison. What suits our purposes best is what could be called
\bfind{initial segment comparison} or just \bfind{left comparison}. (We
leave it to the reader to figure out the analogous definition for the
final segment comparison and to show that this leads to different
results.) For every cut $\Lambda$ in $S$, we define $\Lambda^{\rm L}
\uparrow T$ to be the least initial segment of $T$ containing
$\Lambda^{\rm L}\,$, that is, $\Lambda^{\rm L}\uparrow T$ is the unique
initial segment of $T$ in which $\Lambda^{\rm L}$ forms a cofinal
subset. Then we set
\[\Lambda\uparrow T\>:=\>(\,\Lambda^{\rm L}\uparrow T\,,\,T\setminus
(\Lambda^{\rm L}\uparrow T)\,)\;.\]
Observe that $\Lambda\mapsto\Lambda\uparrow T$ is an order preserving
embedding of the set of cuts of $S$ in the set of cuts in $T$. Now we
can write $\Lambda_1<\Lambda_2$ if $\Lambda_1\uparrow T<\Lambda_2$,
$\Lambda_1=\Lambda_2$ if $\Lambda_1\uparrow T=\Lambda_2$, and
$\Lambda_1>\Lambda_2$ if $\Lambda_1\uparrow T>\Lambda_2\,$. That is,
$\Lambda_1<\Lambda_2$ if $\Lambda_1^{\rm L}$ is contained but not
cofinal in $\Lambda_2^{\rm L}\,$, and $\Lambda_1=\Lambda_2$ if
$\Lambda_1^{\rm L}$ is a cofinal subset of $\Lambda_2^{\rm L}\,$.

\pars
We can embed $S$ in the set of all cuts of $S$ by sending $s\in S$
to the cut
\[s^+\;:=\; (\{t\in S\mid t\leq s\}\,,\,\{t\in S\mid t>s\})\;.\]
We identify $s$ with $s^+$. Then for any cut $\Lambda$, we have $s\leq
\Lambda$ if and only if $s\in\Lambda^{\rm L}$, and equality holds if and
only if $s$ is the maximal element of $\Lambda^{\rm L}$. We also define
\[s^-\;:=\; (\{t\in S\mid t< s\}\,,\,\{t\in S\mid t\geq s\})\;.\]
For any subset $M\subseteq S$, we let $M^+$ denote the cut
\[M^+\>=\>(\{s\in S\mid \exists m\in M: s\leq m\}\,,\,
\{s\in S\mid s>M\})\;.\]
That is, if $M^+=(\Lambda^{\rm L},\Lambda^{\rm R})$ then
$\Lambda^{\rm L}$ is the least initial segment of $S$ which contains
$M$, and $\Lambda^{\rm R}$ is the largest final segment which does not
meet $M$. If $M=\emptyset$ then $\Lambda^{\rm L}=\emptyset$ and
$\Lambda^{\rm R}=M$, and if $M=S$, then $\Lambda^{\rm L}=M$ and
$\Lambda^{\rm R}=\emptyset$. Symmetrically, we set
\[M^-\>=\>(\{s\in S\mid s<M\}\,,\,
\{s\in S\mid \exists m\in M: s\geq m\})\;.\]

\pars
Take two cuts $\Lambda_1=(\Lambda^{\rm L}_1,\Lambda^{\rm R}_1)$ and
$\Lambda_2= (\Lambda^{\rm L}_2,\Lambda^{\rm R}_2)$ in some ordered
abelian group. We let $\Lambda_1+\Lambda_2$ be the cut $(\Lambda^{\rm L},
\Lambda^{\rm R})$ defined by $\Lambda^{\rm L}:=\Lambda^{\rm L}_1+
\Lambda^{\rm L}_2$ (note that the sum of two initial segments is always
an initial segment). This is called the \bfind{left sum} of two cuts;
the \bfind{right sum} is defined by setting $\Lambda^{\rm R}:=
\Lambda^{\rm R}_1+ \Lambda^{\rm R}_2$. In general, left and right sum
are not equal. For instance, the left sum of $0^-$ and $0^+$ is $0^-$,
and the right sum is $0^+$. In this paper, we will only use the left
sum, without mentioning this any further.

We call a cut $\Lambda$ \bfind{idempotent} if $\Lambda+\Lambda=\Lambda$.
Lemma~\ref{charidemp} below will show that in a divisible ordered
abelian group, idempotency of a cut does not depend on whether we take
left or right sums.

We leave the easy proof of the following observation to the reader:
\begin{lemma}
If $G'\supset G$ is an extension of ordered abelian groups, then the
operation $\uparrow$ is an addition preserving embedding of the ordered
set of cuts in $G$ in the ordered set of cuts in $G'$.
\end{lemma}

If $S$ is a subset of an ordered abelian group $G$ and $n\in\N$, we set
$n\cdot S:=\{n\alpha\mid\alpha\in S\}$ and $nS:=\{\alpha_1+\ldots+
\alpha_n \mid\alpha_1,\ldots,\alpha_n\in S\}$. If $\Lambda^{\rm L}$ is
an initial segment, then $n\Lambda^{\rm L}$ is again an initial segment,
and $n\cdot\Lambda^{\rm L}$ is cofinal in $n\Lambda^{\rm L}$. If in
addition $G$ is $n$-divisible, then $n\cdot\Lambda^{\rm L}=n\Lambda^{\rm
L}$. Corresponding assertions hold for $\Lambda^{\rm R}$ in the place of
$\Lambda^{\rm L}$.

In every ordered abelian group, $n\cdot\Lambda^{\rm L}$ and
$n\Lambda^{\rm L}$ define the same cut $n\Lambda:=(n\cdot\Lambda^{\rm
L})^+=(n\Lambda^{\rm L})^+$. Note that $n\Lambda$ coincides with the
$n$-fold (left) sum of $\Lambda$.

\begin{lemma}                                    \label{charidemp}
Let $\Lambda=(\Lambda^{\rm L},\Lambda^{\rm R})$ be a cut in some ordered
abelian group $\Gamma$, and $n>1$ a fixed natural number. The following
assertions are equivalent:
\sn
a) \ $\Lambda$ is idempotent,\nn
b) \ $\Lambda^{\rm L}+\Lambda^{\rm L}=\Lambda^{\rm L}$,\nn
c) \ $i\Lambda = \Lambda$ for every natural number $i>1$,\nn
d) \ $n\Lambda = \Lambda$.
\sn
If $\,\Gamma$ is divisible, then these assertions are also equivalent to
each of the following:
\sn
e) \ $\Lambda^{\rm R}+\Lambda^{\rm R}=\Lambda^{\rm R}$,\nn
f) \ $n\cdot\Lambda^{\rm L}=\Lambda^{\rm L}$,\nn
g) \ $n\cdot\Lambda^{\rm R}=\Lambda^{\rm R}$,\nn
h) \ $\forall\alpha\in \Gamma:\> \alpha\in\Lambda^{\rm L}
\Leftrightarrow n\alpha\in\Lambda^{\rm L}$,\nn
i) \ $\forall\alpha\in \Gamma:\> \alpha\in\Lambda^{\rm R}
\Leftrightarrow n\alpha\in\Lambda^{\rm R}$,\nn
k) \ $\Lambda=H^+$ or $\Lambda=H^-$ for some convex subgroup $H$ of
$\Gamma$.
\end{lemma}
\begin{proof}
The equivalence of a) and b) holds by definition. a) $\Rightarrow$ c) is
proved by induction on $i$. Further, c) $\Rightarrow$ d) is trivial. We
have that $\Lambda\leq 2\Lambda\leq\ldots \leq n\Lambda$ or $\Lambda\geq
2\Lambda\geq\ldots \geq n\Lambda$, depending on whether $0\in
\Lambda^{\rm L}$ or $0\in \Lambda^{\rm R}$. Thus, $n\Lambda= \Lambda$
implies $\Lambda=2\Lambda=\Lambda+\Lambda$; that is, d) implies a).

\pars
Now assume that $\Gamma$ is divisible. Then $\alpha\mapsto n\alpha$
and $\alpha\mapsto \frac{1}{n}\alpha$ are order preserving isomorphisms.
Therefore, f) and g) are equivalent. The equivalence of f) with d) holds
since the divisibility implies that $n\cdot \Lambda^{\rm L}=n
\Lambda^{\rm L}$.
%
%
Further, f) is equivalent to
$n\Lambda^{\rm L}\subseteq\Lambda^{\rm L}\,\wedge\,\Lambda^{\rm L}
\subseteq n \Lambda^{\rm L}$, and this in turn is equivalent to
$n\Lambda^{\rm L}\subseteq\Lambda^{\rm L}\,\wedge \,\frac{1}{n}
\Lambda^{\rm L}\subseteq\Lambda^{\rm L}$. This is equivalent to
$\forall\alpha\in \Gamma:\> \alpha\in\Lambda^{\rm L} \Rightarrow
n\alpha\in\Lambda^{\rm L} \,\wedge\,\alpha\in\Lambda^{\rm L} \Rightarrow
\frac{1}{n}\alpha\in\Lambda^{\rm L}$. But the latter implication can be
reformulated as $n\alpha\in\Lambda^{\rm L}\Rightarrow\alpha\in
\Lambda^{\rm L}$. This proves that f) and h) are equivalent. In the same
way, the equivalence of g) with i) is proved.

As for $\Lambda^{\rm L}$, divisibility also implies that
$n\cdot \Lambda^{\rm R}=n\Lambda^{\rm R}$. Taking $n=2$ in what we have
already proved, we see that e) is equivalent to the $n=2$ case of g),
and hence to a).

Finally, it remains to show the equivalence of k) with the other
conditions. Set
\[H\;:=\;\{\pm \alpha\mid 0\leq \alpha\in \Lambda^{\rm L}\}\,\cup\,
\{\pm \alpha\mid 0\geq\alpha\in \Lambda^{\rm R}\}\;.\]
Note that exactly one of the two sets is empty, depending on whether
$0\in \Lambda^{\rm L}$ or $0\in \Lambda^{\rm R}$. It is easy to see that
$\Lambda=H^+$ if $0\in \Lambda^{\rm L}$ and $\Lambda=H^-$ if $0\in
\Lambda^{\rm R}$. Hence, it suffices to prove that $H$ is a convex
subgroup if and only if $\Lambda^{\rm L}+\Lambda^{\rm L}=\Lambda^{\rm
L}$. Observe that $H$ is always convex and closed under $\alpha\mapsto
-\alpha$. Hence, $H$ is a convex subgroup if and only if it is closed
under addition. In the case of $0\in \Lambda^{\rm L}$, this holds if and
only if $\Lambda^{\rm L}+\Lambda^{\rm L}=\Lambda^{\rm L}$, and in the
case of $0\in \Lambda^{\rm R}$, this holds if and only if $\Lambda^{\rm
R}+\Lambda^{\rm R}=\Lambda^{\rm R}$. But as we have already shown that
b) and e) are equivalent, we see that k) is equivalent with b).
\end{proof}

In a non-divisible group, a condition like $\forall i\in\N:
i\Lambda^{\rm L}=
\Lambda^{\rm L}$ can only hold if $\Lambda^{\rm L}$ is empty, and
condition a) is in general not equivalent to e), h), i), k):
\begin{example}                             \label{exampcut}
Take $\Gamma:=\Z\times\Q$ with the lexicographic ordering, and set
\[\Lambda\>:=\>(\{(m,q)\mid -1\geq m\in\Z\,,\,q\in\Q\}\,,\,
\{(m,q)\mid 0\leq m\in\Z\,,\,q\in\Q\})\;.\]
Then $\Lambda$ satisfies e), h), i), and $\Lambda=H^-$ for the convex
subgroup $H=\{0\}\times\Q$ of $\Gamma$. But $\Lambda$ is not idempotent
since
\sn
$\Lambda+\Lambda\>=$\n
$(\{(m,q)\mid -2\geq m\in\Z\,,\,q\in\Q\}\,,\,
\{(m,q)\mid -1\leq m\in\Z\,,\,q\in\Q\})\><\>\Lambda$.
\sn
On the other hand, the cut induced by $\Lambda$ in the divisible hull
$\tilde{\Gamma}$ of $\Gamma$ is
\sn
$\Lambda\uparrow\tilde{\Gamma}\>=$\n
$(\{(m,q)\mid -1\geq m\in\Q\,,\,q
\in \Q\}\,,\, \{(m,q)\mid -1< m\in\Q\,,\,q\in\Q\})$.
\sn
Since $(-\frac{1}{2},0)$ is in the right cut set while
$2(-\frac{1}{2},0)=(1,0)$ is in the left cut set, this cannot be equal
to $H^+$ or $H^-$ for any convex subgroup $H$ of $\,\tilde{\Gamma}$.
\end{example}

\parm
Take any extension $(L|K,v)$ of valued fields, and $z\in L$. We define
\[\Lambda^{\rm L}(z,K)\;:=\;\{v(z-c)\mid c\in K \mbox{ and } v(z-c)\in
vK\}\;.\]
Further, we set $\Lambda^{\rm R}(z,K):=vK\setminus \Lambda^{\rm L}(z,K)$.

\begin{lemma}
$\Lambda^{\rm L}(z,K)$ is an initial segment of $vK$. Thus,
$\Lambda^{\rm R}(z,K)=\{\alpha\in vK\mid \forall c\in K:\, v(z-c)<
\alpha\}$, and $(\Lambda^{\rm L}(z,K)\,,\,\Lambda^{\rm R}(z,K))=
\Lambda^{\rm L}(z,K)^+$ is a cut in $vK$.
\end{lemma}
\begin{proof}
Take $\alpha\in\Lambda^{\rm L}(z,K)$ and $\beta\in vK$ such that
$\beta<\alpha$. Pick $c,d\in K$ such that $v(z-c)=\alpha$ and
$vd=\beta$. Then $\beta=vd=\min\{vd,v(z-c)\}=v(z-c-d)\in
\Lambda^{\rm L}(z,K)$. This proves that
$\Lambda^{\rm L}(z,K)$ is an initial segment of $vK$.
\end{proof}

We have seen in Lemma~\ref{charidemp} that in a divisible ordered
abelian group we have many nice characterizations of idempotent cuts; in
particular, we are very interested in characterization k) which shows
that idempotent cuts correspond to upper or lower edges of convex
subgroups. This is the reason for the following definition. Take an
element $z$ in any valued field extension $(L,v)$ of $(K,v)$. Then
the \bfind{distance of $z$ from $K$} is the cut
\[\dist(z,K)\;:=\;(\Lambda^{\rm L}(z,K),
\Lambda^{\rm R}(z,K))\uparrow\widetilde{vK}\]
in the divisible hull $\widetilde{vK}$ of $vK$.

\parm
Take two elements $y,z$ in some valued field extension of $(K,v)$. We
define
\[z\,\sim_K\, y\]
to mean that $v(z-y)>\dist(z,K)$. Note that by our identification of the
value $v(z-y)$ with the cut $v(z-y)^+$,
\[v(z-y)>\dist(z,K)\mbox{ \ \ if and only if \ \ }
v(z-y)>\Lambda^{\rm L}(z,K)\;.\]

\begin{lemma}                               \label{sim1}
\mbox{ }\n
1) \ If $z\sim_K y$ then $v(z-c)=v(y-c)$ for all $c\in K$ such that
$v(z-c)\in vK$, whence $\Lambda^{\rm L}(z,K)=\Lambda^{\rm L}(y,K)$ and
$\dist(z,K)=\dist(y,K)$.
\sn
2) \ If $\Lambda^{\rm L}(z,K)$ has no maximal element, then the
following are equivalent:
\n
a) \ $z\sim_K y$,
\n
b) \ $v(z-c)=v(y-c)$ for all $c\in K$ such that $v(z-c)\in vK$,
\n
c) \ $v(z-y)\geq\dist(z,K)$.
\end{lemma}
\begin{proof}
1): \ Assume that $z\sim_K y$.
If $v(z-c)\in vK$, then $v(z-c)\in\Lambda^{\rm L}(z,K)$ and therefore,
$v(z-y)>v(z-c)$. Hence, $v(y-c)=\min\{v(z-c),v(z-y)\}=v(z-c)$.
\sn
2): \ The implication a)$\Rightarrow$b) follows from part 1). Now assume
that b) holds, and take any $c\in K$ such that $v(z-c)\in vK$. Then
because of $v(z-c)=v(y-c)$, we obtain $v(z-y)\geq\min\{v(z-c),v(y-c)\}
=v(z-c)$. This shows that $v(z-y)\geq\dist(z,K)$. We have proved that b)
implies c).

Since $\Lambda^{\rm L}(z,K)$ has no maximal element, $v(z-c)$ cannot be
the maximal element of $\Lambda^{\rm L}(z,K)$. Thus, $v(z-y)\geq\dist(z,K)$
implies $v(z-y)>\dist(z,K)$, which proves the implication
c)$\Rightarrow$a).
\end{proof}

\begin{lemma}                               \label{sim2}
If $(K,v)\subseteq (L,v)\subseteq (L(z),v)$, then
\begin{equation}                            \label{dged}
\mbox{\rm dist}(z,L)\>\geq\> \mbox{\rm dist}(z,K)\;.
\end{equation}
If ``$\,>$'' holds, then there exists an element $y\in L$ such that
$z\sim_K y$.
\end{lemma}
\begin{proof}
Since $K\subseteq L$, we have that $\Lambda^{\rm L}(z,K)\subseteq
\Lambda^{\rm L}(z,L)$, whence (\ref{dged}). If ``$\,>$'' holds, then
there exists an element $y\in L$ such that $v(z-y)>\Lambda^{\rm L}
(z,K)$, i.e., $z\sim_K y$.
\end{proof}

We define
\[v(z-K)\;:=\;\{v(z-c)\mid c\in K\}\;=\;v\im_K(X-z)\]
Note that
\[\Lambda^{\rm L}(z,K)=v(z-K)\cap vK\>.\]
Hence if $vK(z)=vK$, then $\Lambda^{\rm L}(z,K)=v(z-K)$.

\begin{theorem}   {\rm\ \ \ (cf.\ Theorem 1 of [Ka])} \label{KT1}\n
Let $L$ be an immediate extension of $K$. Then for every element $z\in
L\setminus K$ it follows that $v(z-K)$ has no
maximal element and that $v(z-K)=\Lambda^{\rm L}(z,K)$.
In particular, $vz$ is not maximal in $\Lambda^{\rm L}(z,K)$ and
therefore, $vz<\mbox{\rm dist}(z,K)$.
\end{theorem}
\begin{proof}
Take $z\in L\setminus K$. Then $\infty\notin v(z-K)$. If $(L|K,v)$ is
immediate, then $vL=vK$ and therefore, $v(z-K)= \Lambda^{\rm L}(z,K)$.
Take any $c\in K$. Then $v(z-c)\in vL=vK$ and thus there exists $d\in K$
such that $vd(z-c)=0$. So $d(z-c)v\in Lv=Kv$. Hence, there exists $d'\in
K$ such that $(d(z-c)-d')v = 0$, which means that $v(z-c-d'd^{-1})> -vd
=v(z-c)$. Since $c+d'd^{-1}\in K$ and $v(z-c-d'd^{-1})\in vL=vK$, this
shows that $v(z-c)$ was not the maximal element of $v(z-K)$. This proves
that $v(z-K)$ has no maximal element.
\end{proof}


The following is a corollary to Lemma~\ref{apCs}:
\begin{corollary}                           \label{exoptappr}
If $(K(z)|K,v)$ is an algebraic extension and $(K,v)$ is
algebraically maximal, then $v(z-K)$ has a maximum.
\end{corollary}
\begin{proof}
If $v(z-K)$ has no maximum, then there is a sequence
$(c_\nu)_{\nu<\lambda}$ without last element (so $\lambda$ is a limit
ordinal) and such that $(v(z-c_\nu))_{\nu<\lambda}$ is strictly
increasing and cofinal in $v(z-K)$. The former implies that
$(c_\nu)_{\nu<\lambda}$ is a pseudo Cauchy sequence with $z$ as a limit.
The latter implies that $(c_\nu)_{\nu<\lambda}$ has no limit in $K$
since by Lemma~3 of [Ka], any limit $b$ satisfies $v(z-b)>v(z-c_\nu)$
for all $\nu<\lambda$. By Lemma~\ref{apCs}, $(c_\nu)_{\nu<\lambda}$ is
of algebraic type. Hence by Theorem~3 of [Ka], there is a non-trivial
immediate algebraic extension of $K$, which shows that $K$ cannot be
algebraically maximal.
\end{proof}

Does the converse of Theorem~\ref{KT1} also hold, that is, if $v(z-K)$
has no maximal element, is then the extension $(K(z)|K,v)$ immediate?
This is far from being true. Under certain additional conditions
however, the converse holds:
\begin{lemma}                         \label{ueGp1}  
Take an extension $(K(z)|K,v)$ of valued fields of degree
$p=\mbox{\rm char}(Kv)$ and such that the extension of $v$ from $K$
to $K(z)$ is unique.
\n
1) \ If $v(z-K)$ has no maximal element, then $(K(z)|K,v)$ is immediate.
\n
2) \ If $(K(y)|K,v)$ is an immediate extension and if $y\sim_K z$
in some common valued extension field of $K(z)$ and $K(y)$, then
$(K(z)|K,v)$ is also an immediate extension.
\end{lemma}
\begin{proof}
1): \ Since the extension of $v$ from $K$ to $K(z)$ is unique, we have
\[p\;=\;[K(z):K]\;=\;(vK(z):vK)[K(z)v:Kv]\,p^\nu\]
by (\ref{LoO}). Assume that $(K(z)|K,v)$ is not immediate. Then
$(vK(z):vK)=p$ or $[K(z)v:Kv]=p$. If $(vK(z):vK)=p$, then we can choose
$b_1,\ldots,b_p\in K(z)$ such that the values $vb_1,\ldots,vb_p$ belong
to distinct cosets modulo $vK$. If $[K(z)v:Kv]=p$, then we can choose
$b_1,\ldots,b_p\in K(z)$ such that $vb_1=\ldots=vb_p=0$ and the residues
$b_1v,\ldots,b_p v$ form a basis of $K(z)v|Kv$. In both cases, we obtain
that $b_1,\ldots,b_p$ is a valuation basis of $(K(z)|K,v)$. In the first
case, given any $c_1,\ldots,c_p\in K$, this follows from the ultrametric
triangle law since all values $vc_ib_i\,$, $1\leq i\leq p$, must be
distinct. In the second case, we may assume w.l.o.g.\ (after suitable
renumbering) that $vc_1=\min_i vc_i\,$; then $vc_1^{-1}c_ib_i \geq 0$
and we obtain
\[\left(\sum_{i=1}^{p}c_1^{-1}c_ib_i\right)v\>=\> b_1v+\sum_{i=2}^{p}
(c_1^{-1}c_i)v\,b_iv\>\ne\>0\;.\]
This yields $v\sum_{i=1}^{p}c_1^{-1}c_ib_i=0$ and thus
$v\sum_{i=1}^{p}c_ib_i=vc_1=\min_i vc_i=\min_i vc_ib_i\,$.

Without loss of generality, we can choose $b_1=1$. We write
$z=\sum_{i=1}^{p}c_ib_i\,$. For every $c\in K$, we obtain
\[v(z-c) \;=\;\min \{v(c_1-c),vc_2,\ldots,vc_p\}\;\leq\;
\min \{vc_2,\ldots,vc_p\}\;.\]
The maximum value $\min \{vc_2,\ldots,vc_p\}$ is assumed for $c_1-c=0$.
That is, $v(z-c_1)$ is the maximum of $v(z-K)$.
\sn
2): \ Assume that $(K(y)|K,v)$ is an immediate extension. Then by
Theorem~\ref{KT1}, $v(y-K)=\Lambda^{\rm L}(y,K)$ has no maximal
element. By Lemma~\ref{sim1}, $z\sim_K y$ implies that $v(y-c)=v(z-c)$
for all $c\in K$ such that $v(y-c)\in vK$, that is, for all $c\in K$. It
follows that $v(z-K)$ has no maximal element. Now part 1) shows that
$(K(z)|K,v)$ is an immediate extension.
\end{proof}

\begin{lemma}                               \label{fasimfbh}
Assume that $(K,v)$ is henselian, that $(K(z)|K,v)$ is an immediate
extension, and that $z\sim_K y$ in some common valued extension field of
$K(z)$ and $K(y)$. Take a polynomial $f\in K[X]$ of degree smaller than
$p=\chara Kv$. Then $f(z)\sim_K f(y)$.
\end{lemma}
\begin{proof}
Let $f\in K[X]$ be a polynomial of degree $<p$. Since $(K(z)|K,v)$ is
immediate and $f(z)\in K(z)$, we know from Theorem~\ref{KT1} that
$\Lambda^{\rm L}(f(z),K)$ has no maximal element. Hence by part 2) of
Lemma~\ref{sim1} it suffices to show that $v(f(z)-c)= v(f(y)-c)$ for all
$c\in K$. Since $f-c$ is again of degree $<p$, we see that it suffices
to show that $vf(z)=vf(y)$ for all polynomials $f$ of degree $<p$.

Again by Theorem~\ref{KT1} we know that $\Lambda^{\rm L}(z,K)$ has no
maximal element. Since $z\sim_K y$, part 1) of Lemma~\ref{sim1} shows
that $\Lambda^{\rm L}(z,K)=\Lambda^{\rm L}(y,K)$. As in the proof of
Corollary~\ref{exoptappr} we find a pseudo Cauchy sequence
$(c_\nu)_{\nu<\lambda}$ in $(K,v)$ that has both $z$ and $y$ as a limit,
but no limit in $K$. For every polynomial of degree $<p$, the value of
the sequence $(f(c_\nu))_{\nu<\lambda}$ must eventually be fixed since
otherwise, Theorem~3 of [Ka] would show the existence of an immediate
extension of $(K,v)$ of degree less than $p$. But since $(K,v)$ is
henselian, the Lemma of Ostrowski shows that this is impossible, since
the defect must be a power of $p$. Now one shows like in the proof of
Theorem~2 of [Ka] that both $vf(z)$ and $vf(y)$ are equal to the
eventually fixed value of the sequence $(f(c_\nu))_{\nu<\lambda}$.
\end{proof}

We will show that we can drop the condition that $(K,v)$ be henselian if
the element $y$ is purely inseparable over $K$. To this end, we need the
following result which is proved in [Ku7]:
\begin{lemma}
Let $(K,v)$ be a valued field, $K^h$ its henselization w.r.t.\ a fixed
extension of $v$ to the algebraic closure $\tilde{K}$, and $y\in
\tilde{K}$. If
\[\dist(y,K^h)\>>\>\dist(y,K)\;,\]
then $y$ is not purely inseparable over $K$.
\end{lemma}

\begin{lemma}                               \label{fasimfb}
Assume that $(K(z)|K,v)$ is an immediate extension and that $z\sim_K y$
in some common valued extension field of $K(z)$ and $K(y)$. Suppose that
$y$ is purely inseparable over $K$. Take a polynomial $f\in K[X]$ of
degree smaller than $p=\chara Kv$. Then $f(z)\sim_K f(y)$.
\end{lemma}
\begin{proof}
Since henselizations are immediate extensions and since $K^h(z)$ lies
in the henselization of $K(z)$, we know that $(K^h(z)|K^h,v)$ is an
immediate extension. From the previous lemma we infer that
$\dist(y,K^h)= \dist(y,K)$. Therefore, $z\sim_K y$ implies that
$z\sim_{K^h} y$. From Lemma~\ref{fasimfbh} we obtain that
$f(z)\sim_{K^h} f(y)$, whence $f(z)\sim_K f(y)$.
\end{proof}

\parm
If $\alpha\in vK$ and $\Lambda$ is a cut in $vK$, then $\alpha+\Lambda:=
(\alpha+\Lambda^{\rm L}\,,\,\alpha+\Lambda^{\rm R})$. Since addition of
$\alpha$ is an order preserving isomorphism of $vK$, this is again a
cut. The proof of the following lemma is straightforward:
\begin{lemma}                  \label{aat}
For every $c\in K$,
\begin{eqnarray*}
\Lambda^{\rm L}(z+c,K) \;=\; \Lambda^{\rm L}(z,K) &\mbox{ and }&
\mbox{\rm dist}(z+c,K) \;=\; \mbox{\rm dist}(z,K)\;,\\
\Lambda^{\rm L}(cz,K) \;=\; vc+\Lambda^{\rm L}(z,K) &\mbox{ and }&
\mbox{\rm dist}(cz,K) \;=\; vc+\mbox{\rm dist}(cz,K)\;,\\
z\,\sim_K\, y & \Rightarrow & z+c \>\sim_K\> y+c \;,\\
c\ne 0\>\wedge\> z\,\sim_K\, y & \Rightarrow & cz \,\sim_K\, cy \;.
\end{eqnarray*}
\end{lemma}

%
%
%
\subsection{Properties of Artin-Schreier extensions}
In this section, we collect a few facts about Artin-Schreier extensions
of valued fields. Throughout this section, we assume that
$K(\vartheta)|K$ is an Artin-Schreier extension of degree $p$ with
$\vartheta^p-\vartheta=a\in K$.

\begin{lemma}                               \label{allASg}
If $\chara K=p$, then $\vartheta'$ is another Artin-Schreier generator
of $L|K$ if and only if $\vartheta' =i\vartheta +c$ for some
$i\in\{1,\ldots,p-1\}$ and $c\in K$.
\end{lemma}
\begin{proof}
If $\vartheta,\vartheta'$ are roots of the same polynomial $X^p-X-a$,
then $\vartheta-\vartheta'$ is a root of $X^p-X$, whose roots are
$0,1,\ldots,p-1\in\F_p\,$. Hence, $\vartheta,\vartheta+1,\ldots,
\vartheta+p-1$ are all roots of $X^p-X-a$. Pick a non-trivial $\sigma\in
\Gal L|K$. We then have that $\sigma\vartheta-\vartheta=j$ for some
$j\in\F_p^\times$. If $\vartheta,\vartheta'$ are any two Artin-Schreier
generators of $L|K$ such that $\sigma\vartheta-\vartheta=
\sigma\vartheta'-\vartheta'$, then we have $\sigma(\vartheta-\vartheta')
=\vartheta-\vartheta'$. Since $\sigma$ is a generator of $\Gal L|K\isom
\Z/p\Z$, it follows that $\tau(\vartheta-\vartheta')=\vartheta-
\vartheta'$ for all $\tau\in\Gal L|K$, that is, $\vartheta-\vartheta'
\in K$. If $\vartheta,\vartheta'$ are any two Artin-Schreier generators
of $L|K$ such that $\sigma\vartheta-\vartheta=j\in\F_p^\times$ and
$\sigma\vartheta'-\vartheta'=j'\in\F_p^\times$, then there is some $i\in
\{1,\ldots,p-1\}$ such that $ij=j'$ and therefore, $\sigma i\vartheta-
i\vartheta =ij=j'$. Then by what we have shown before, $\vartheta'=i
\vartheta+c$ for some $c\in K$.

Conversely, if $\vartheta$ is an Artin-Schreier generator of $L|K$ and
if $i\in\{1,\ldots,p-1\}$ and $c\in K$, then $(i\vartheta+c)^p-(i
\vartheta+c)=i(\vartheta^p-\vartheta)+c^p-c\in K$. But $i\vartheta+c$
cannot lie in $K$, so $K(i\vartheta+c)=L$ since $[L:K]$ is a prime. This
shows that also $i\vartheta+c$ is an Artin-Schreier generator of $L|K$.
\end{proof}

We will frequently use the following easy observation:
\begin{lemma}                               \label{valASgen}
If $va\leq 0$, then $v\vartheta=\frac{1}{p}va$, and if $va\geq 0$,
then $v\vartheta=va$.
\end{lemma}
\begin{proof}
We have that $\vartheta^p-\vartheta=a$. If $v\vartheta\ne 0$, then
$v\vartheta^p=pv\vartheta\ne v\vartheta$ and therefore, $va=
v(\vartheta^p-\vartheta)=\min\{pv\vartheta, v\vartheta\}\ne 0$ by the
ultrametric triangle law. If $va=0$, we thus have $v\vartheta=0$. For
$va<0$ we must have $v\vartheta <0$, whence $pv\vartheta<v\vartheta$ and
$va=pv\vartheta$. For $va>0$ we must have $v\vartheta >0$, whence
$v\vartheta<pv\vartheta$ and $va=v\vartheta$.
\end{proof}

The following lemma gives a first classification of Artin-Schreier
extensions of valued fields.

\begin{lemma}                          \label{classASE}
Assume that $\chara Kv=p$. If $va>0$ or if $va=0$ and $X^p-X-av$ has a
root in $Kv$, then $\vartheta$ lies in the henselization of $K$ (with
respect to every extension of the valuation to the algebraic closure of
$K$) and there are precisely $p$ many distinct extensions of $v$ from
$K$ to $K(\vartheta)$; hence, equality holds in the fundamental
inequality (\ref{fundineq}).

If $va=0$ and $X^p-X-av$ has no root in $Kv$, then $K(\vartheta)v|Kv$ is
a separable extension of degree $p$ and $(K(\vartheta)|K,v)$ is
defectless.

If $(K(\vartheta)|K,v)$ has non-trivial defect, then $va<0$.
\end{lemma}
\begin{proof}
If $va>0$, then the reduction of $X^p-X-a$ modulo $v$ is $X^p-X$ which
splits completely in $Kv$ and has $p$ many distinct roots since $\chara
Kv=p>0$. Then by Hensel's Lemma, $X^p-X-a$ splits completely in every
henselization of $K$.

If $va=0$ and $X^p-X-av$ has a root in $Kv$, then $X^p-X-av$ splits
completely in $Kv$  and has $p$ many distinct roots. Hence again,
$X^p-X-a$ splits completely in every henselization of $K$.

In both cases, pick one extension of $v$ to $K(\vartheta)$ and
call it again $v$. The roots of $X^p-X-a$ are in one-to-one
correspondence with the roots of $X^p-X-av$. Hence, the roots
$\eta_1,\ldots,\eta_p$ of $X^p-X-a$ have distinct residues in
$Kv$, say, $c_1v,\ldots,c_pv$ with $c_i\in K$. If $\sigma_i$ is the
automorphism of $K(\vartheta)|K$ which sends $\eta_1$ to $\eta_i\,$,
then $v\circ\sigma_i(\eta_1-c_i)=v(\sigma_i\eta_1-c_i)=v(\eta_i-c_i)>0$
and $v\circ\sigma_j(\eta_1-c_i)=v(\eta_j-c_i)=0$ for $j\ne i$. This
shows that the extensions $v\circ\sigma_i$ are distinct for $1\leq
i\leq p$. Since all extensions are conjugate and therefore must be of
the form $v\circ\sigma_i\,$, we find that there are
precisely $p$ many distinct extensions.

If $va=0$ and $X^p-X-av$ has no root in $Kv$, then $[Kv(\vartheta v):Kv]
=p$ since $\vartheta v$ is a root of $X^p-X-av$ and this polynomial is
irreducible over $Kv$. We obtain that $p=[K(\vartheta):K]\geq
[K(\vartheta)v:Kv]\geq [Kv(\vartheta v):Kv]=p$ and see that equality
must hold everywhere. So $K(\vartheta)v=Kv(\vartheta v)$ is a separable
extension of $Kv$, and $K(\vartheta)|K$ is defectless.

By what we have proved, $K(\vartheta)|K$ is defectless whenever
$va\geq 0$. This yields the last assertion of our lemma.
\end{proof}

If $\chara K=p>0$, then the Artin-Schreier polynomial is additive.
If $\vartheta$ is a root of $X^p-X-a$ and if $c\in K$, then
\[(\vartheta-c)^p\,-\,(\vartheta-c)\;=\;\vartheta^p-\vartheta-c^p+c\;=\;
a-c^p+c\;,\]
that is, $\vartheta-c$ is a root of the polynomial $X^p-X-(a-c^p+c)$.

\begin{remark}
Since $K(\vartheta)=K(\vartheta-c)$, this shows that $p$-th powers
appearing in $a$ can be replaced by their $p$-th roots without changing
the extension. This allows to deduce normal forms for $a$ that serve
various purposes. They are key tools in [Ku4] and [Ku8] and in related
work of S.~Abhyankar and H.~Epp.
\end{remark}

\begin{corollary}                                   \label{wpC}
Assume that $\chara K=p$ and that $(K(\vartheta)|K,v)$ has non-trivial
defect. Then $v(\vartheta-c)<0$ for every $c\in K$, and consequently,
$\mbox{\rm dist}(\vartheta,K)\leq 0^-$.
\end{corollary}
\begin{proof}
If there exists $c\in K$ such that $v(\vartheta - c)\geq 0$, then
$\vartheta - c$ is a root of the polynomial $X^p-X-(a-c^p+c)$ and by
Lemma~\ref{valASgen} we have $v(a-c^p+c)=v(\vartheta-c)\geq 0$. But then
by Lemma~\ref{classASE}, the field $K(\vartheta)=K(\vartheta-c)$ cannot
be a defect extension of $K$.
\end{proof}

The converse is also true, in the following sense:
\begin{lemma}                               \label{uniqextv}
Assume that $\chara K=p$. If $\mbox{\rm dist}(\vartheta,K)\leq 0^-$ and
$v(\vartheta-K)$ has no maximal element, then the extension of $v$ from
$K$ to $K(\vartheta)$ is unique, $(K(\vartheta)|K,v)$ is immediate and
consequently, $K(\vartheta)|K$ is an Artin-Schreier defect extension.
\end{lemma}
\begin{proof}
In [Ku7] we show that the assumption that $\mbox{\rm dist}(\vartheta,K)
\leq 0^-$ implies that the extension of $v$ from $K$ to $K(\vartheta)$
is unique. Since $v(\vartheta-K)$ has no maximal element,
Lemma~\ref{ueGp1} yields that $(K(\vartheta)|K,v)$ is immediate.
\end{proof}

\pars
We will also need the following fact:
\begin{lemma}                               \label{AScpieASc}
Let $K$ be an Artin-Schreier closed field of characteristic $p>0$. Then
also every purely inseparable extension of $K$ is Artin-Schreier closed.
\end{lemma}
\begin{proof}
If $\chara K=0$ then every purely inseparable extension is trivial and
there is nothing to show. So let $\chara K=p>0$. Assume $L$ to be a
purely inseparable extension of the Artin-Schreier closed field $K$.
Take $a\in L$ and let $\vartheta\in\tilde{L}$ be a root of $X^p-X-a$.
Let $m\geq 0$ be the minimal integer such that $a^{p^m}\in K$. Then
$(\vartheta^{p^m})^p-\vartheta^{p^m}=(\vartheta^p-\vartheta)^{p^m}=
a^{p^m}$. Since $K$ is Artin-Schreier closed by assumption, it follows
that $\vartheta^{p^m}\in K$. The field $K(\vartheta)$ contains $a=
\vartheta^p-\vartheta$ and thus, $[K(\vartheta):K]\geq [K(a):K]=p^m$. On
the other hand, $p^m\geq [K(\vartheta):K]$ since $\vartheta^{p^m}\in K$.
Consequently, $[K(\vartheta):K]= [K(a):K]$, showing that $\vartheta \in
K(a)\subseteq L$.
\end{proof}


%
%
%
\section{Inseparably defectless fields}     \label{sectidf}
In this section, we shall give a characterization of inseparably
defectless fields. Throughout, we assume that $\chara K=p$.
Recall that every purely inseparable algebraic
extension admits a unique extension of the valuation. Every defectless
field and in particular every trivially valued field is inseparably
defectless. Note that a valued field can be \bfind{inseparably maximal},
that is, it does not admit proper immediate purely inseparable
extensions, without being inseparably defectless. The field $(F,v)$ of
Example~3.25 in [Ku5] is of this kind.

Let us observe that for an inseparably defectless field $(K,v)$,
every immediate extension is separable. Indeed, it follows from
Lemma~\ref{ivd} that every immediate extension of $(K,v)$ is
linearly disjoint from the defectless extension $(K^{1/p^{\infty}}|
K,v)$. In the literature, one can find the expression ``excellent''
for those fields for which all immediate extensions are separable
(cf.\ \fvklit{DEL1}, D\'efinition 1.41). But there are also other
properties of certain valuation rings for which this expression is used.

By definition, $(K,v)$ is an inseparably defectless field if and
only if the extension $(K^{1/p^{\infty}}|K,v)$ is defectless. For this
to hold, it is already sufficient that $(K^{1/p}|K,v)$ is defectless:
\begin{lemma}                               \label{charid1/p}
The field $(K,v)$ is inseparably defectless if and only if
$(K^{1/p}|K,v)$ is defectless, and this holds if and only
if $(K|K^p,v)$ is defectless.
\end{lemma}
\begin{proof}
The first implication ``$\Rightarrow$'' is trivial. Assume that
$(K^{1/p}|K,v)$ is defectless. The Frobenius endomorphism sends the
extension
\[(K^{1/p^2}| K^{1/p},v)\]
onto the extension $(K^{1/p}|K,v)$
and is valuation preserving. Consequently, also the former extension is
defectless. By induction, we find that $(K^{1/p^m}| K^{1/p^{m-1}},v)$ is
defectless for every $m\geq 1$. By a repeated application of
Lemma~\ref{transdl}, also $(K^{1/p^m}|K,v)$ is defectless. Since every
finite subextension of $K^{1/p^{\infty}}|K$ is already contained in
$K^{1/p^m}$ for some $m$, it follows that $(K^{1/p^{\infty}}|K,v)$ is
defectless.

The second equivalence is proved again by use of the Frobenius
endomorphism.
\end{proof}

\begin{lemma}
The field $(K,v)$ is inseparably defectless if and only if for every
finite (possibly trivial) subextension $L|K^p$ of $K|K^p$ and every
subextension $L(b)|L$ of $K|L$ of degree $p$, the set $v(b-L)$ has a
maximal element.
\end{lemma}
\begin{proof}
By the previous lemma, $(K,v)$ is inseparably defectless if and only if
every finite subextension $(F|K^p,v)$ of $(K|K^p,v)$ is defectless. But
$F|K$ is a tower of purely inseparable extensions of degree $p=\chara
K$, so the latter holds if and only if each extension in the tower is
defectless. So by a repeated application of Lemma~\ref{md}
we see that $(K,v)$ is inseparably defectless if and only
if for every finite subextension $L|K^p$ of $K|K^p$ and every
subextension $L(b)|L$ of $K|L$ of degree $p$, the extension $(L(b)|L,v)$
is defectless. The latter is equivalent to $v(b-L)$ having a maximal
element. Indeed, if $(L(b)|L,v)$ is a defect extension, then it is
immediate and by Theorem~\ref{KT1}, $v(b-L)$ has no maximal element. The
converse holds by part 1) of Lemma~\ref{ueGp1} since the extension of
$v$ from $L$ to $L(b)$ is unique.
\end{proof}

Now we are able to give the
\sn
{\bf Proof of Theorem~\ref{extrid}}: \ The first assertion is an easy
consequence of the last lemma. Given $L$ and $b$ as in that lemma, we
take $b_1\ldots,b_n$ to be a $K^p$-basis of $L$. Then $c\in L$ if and
only if $c=\sum_{i=1}^{n} b_ic_i^p$ for some $c_1,\ldots,c_n\in K$.
Hence, $v(b-L)$ has a maximum if and only if $(K,v)$ is $K$-extremal
with respect to the polynomial (\ref{b-bixp}).

To prove the second assertion of Theorem~\ref{extrid}, we assume that
$vK$ is divisible or a $\Z$-group. The same is then true for $vK^p$ and
for $vL$ for every $L$ as in the previous lemma. Further, we note that
in the previous lemma, we can restrict the scope to all $b\in {\cal O}$.
As well, we can restrict the scope to all defectless extensions
$(L|K^p,v)$. So we can choose $b_1,\ldots,b_n$ to be a valuation basis
of $(L|K^p,v)$. If $vK^p$ is divisible, then $vL=vK^p$ and we can
assume in addition that $vb_1=\ldots=vb_n=0$. If $vK^p$ is a $\Z$-group
with least positive element $\alpha$, then we can assume in addition
that for $1\leq i\leq n$, $vb_i= \frac{\ell_i}{p^m}\alpha$ for some
$\ell_i\in\{0,\ldots,p^m-1\}$, with $m\geq 0$ fixed; so $0\leq vb_i<
\alpha$. Now it remains to show that $v(b-L)$ has a maximal element if
and only if $(K,v)$ is ${\cal O}$-extremal with respect to the
polynomial (\ref{b-bixp}). We observe that $0\leq v(b-0)\in v(b-L)$.
Take $c\in L$ such that $v(b-c)\geq 0$. We write $c=\sum_{i=1}^{n}
b_ic_i^p$ with $c_1,\ldots,c_n\in K$. It follows that
\[0\>\leq\>vc\>=\>v\sum_{i=1}^{n} b_ic_i^p\>=\>
\min_i vb_ic_i^p\>.\]
Hence for $1\leq i\leq n$, $vb_i+vc_i^p\geq 0$, and by our assumptions
on the values $vb_i$, this implies that $vc_i^p\geq 0$ and hence $c_i\in
{\cal O}$. This shows that the image of ${\cal O}^n$ under the
polynomial (\ref{b-bixp}) is a final segment of $v(b-L)$, hence one of
the sets has a maximal element if and only if the other has.    \QED

\begin{corollary}                           \label{extr2}
Every extremal field with value group a divisible or a $\Z$-group is
inseparably defectless.
\end{corollary}

\begin{corollary}                           \label{idelem}
The property ``inseparably defectless'' is elementary in the language of
valued fields.
\end{corollary}
\begin{proof}
The property can be axiomatized by an infinite scheme of axioms where
$n$ runs through all powers $p^\nu$ of $p$. Each of the axioms
quantifies over all $b\in K$ and all bases of finite extensions of
$K^p$. The latter is done by quantifying over all choices of
$a_1,\ldots,a_\nu\in K$ such that the elements $a_1^{e_1}\cdot
\ldots\cdot a_\nu^{e_\nu}$, $0\leq e_i<p$ are linearly independent over
$K^p$ (which can be expressed by an elementary sentence). Also the
additional conditions concerning the values of these elements and that
they form a valuation basis are elementary in the language of valued
fields.
\end{proof}

For a valued field of finite $p$-degree, one knows several properties
which are equivalent to ``inseparably defectless''. The following
theorem is due to F.~Delon [D1]:         

\begin{theorem}                             \label{charinsdl}
Let $K$ be a field of characteristic $p>0$ and finite $p$-degree
$[K:K^p]$. Then for the valued field $(K,v)$, the property of being
inseparably defectless is equivalent to each of the following
properties:\sn
{\bf a)}\ \ $[K:K^p]=(vK:pvK)[Kv:Kv^p]$, i.e.,
$(K|K^p,v)$ is a defectless extension\sn
{\bf b)}\ \ $(K^{1/p}|K,v)$ is a defectless extension\sn
{\bf c)}\ \ every immediate extension of $(K,v)$ is separable\sn
{\bf d)}\ \ there is a separable maximal immediate extension of
$(K,v)$.
\end{theorem}
\begin{proof}
The equivalence of ``$(K,v)$ inseparably defectless'' with properties a)
and b) follows readily from Lemma~\ref{charid1/p}. By Lemma~\ref{ivd}
every immediate extension of an inseparably defectless field is
linearly disjoint from $K^{1/p^\infty}|K$, i.e., it is separable. This
proves that ``$(K,v)$ inseparably defectless'' implies property c).
Since every valued field admits a maximal immediate extension (cf.\ [Kr]
and [G]), it follows that property c) implies property d).

It now suffices to show that property d) implies property a). Let
$(L,v)$ be a separable maximal immediate extension of $(K,v)$. The
separability implies that $[L:L^p]=[L^{1/p}:L]\geq [L.K^{1/p}:L]=
[K^{1/p}:K]=[K:K^p]$. On the other hand, we have that $vL=vK$ and
$Lv=Kv$. Since $(L,v)$ is a maximal immediate extension, it is a maximal
field. Since every maximal field is a defectless field (cf.\ [W],
Theorem~31.21), the extension $(L^{1/p}|L,v)$ is defectless, and by
Lemma~\ref{charid1/p} we conclude that also $(L|L^p,v)$ is defectless.
Since $(vL:pvL)= (vK:pvK)$ and $[Lv:Lv^p]=[Kv:Kv^p]$ are finite (as
$[K:K^p]$ is finite), it follows that $[L:L^p]$ is finite and equal to
$(vL:pvL)[Lv:Lv^p]$. Consequently,
\begin{eqnarray*}
[L:L^p] & = & (vL:pvL)[Lv:Lv^p]\>=\>(vK:pvK)[Kv:Kv^p]\\
 & \leq & [K:K^p]\>\leq\> [L:L^p]\;.
\end{eqnarray*}
Thus, equality holds everywhere, showing that a) holds.
\end{proof}

From the proof, we also obtain:
\begin{corollary}                           \label{p-degmax}
A given maximal immediate extension of a valued field $(K,v)$ has the
same $p$-degree as $K$ if and only if $(K,v)$ is an inseparably
defectless field.
\end{corollary}

The very useful upward direction of the following lemma was also
stated by F.~Delon ([D1], Proposition 1.44):         
\begin{lemma}                               \label{delon}
Let $(L|K,v)$ be a finite extension of valued fields. Then $(K,v)$ is
inseparably defectless and of finite $p$-degree if and only if $(L,v)$
is.
\end{lemma}
\begin{proof}
The $p$-degree of a field does not change under finite extensions.
Assume that one and hence both fields have finite $p$-degree. Since
$[L:K]$ is finite, also $(vL:vK)$ and $[Lv:Kv]$ are finite.
%
%
Hence, also the $p$-degree of $Lv$ is equal to that of $Kv$. The same
can be shown for ordered abelian groups: $(vL:pvL) =(vK:pvK)$ (the
details are left to the reader). It follows that
$[K:K^p]=(vK:pvK)[Kv:Kv^p]$ if and only if $[L:L^p]=(vL:pvL)[Lv:Lv^p]$,
which by Theorem~\ref{charinsdl} means that $(K,v)$ is inseparably
defectless if and only if $(L,v)$ is.
\end{proof}
\n
In Lemma~\ref{l5} in Section~\ref{sectth1} we will generalize the upward
direction to the case of arbitrary $p$-degree.

%
%
\section{Artin-Schreier defect extensions}               \label{+4}
%

%
%
\subsection{Classification of Artin-Schreier defect extensions}
We will consider the following situation:
\sn
$\bullet$ \ $(L|K,v)$ an Artin-Schreier defect extension of valued
fields of characteristic $p>0$,
\n
$\bullet$ \ $\vartheta\in L\setminus K$ an Artin-Schreier generator of
$L|K$,
\n
$\bullet$ \ $a= \wp (\vartheta) = \vartheta ^p - \vartheta \in K$,
\n
$\bullet$ \ $\delta=\dist(\vartheta, K)$.
\sn
Since $(L|K,v)$ is immediate and non-trivial, we know that
$v(\vartheta-K)=\Lambda^{\rm L}(\vartheta,K)$ has no maximal element and
that $\delta > v\vartheta$ (cf.\ Theorem~\ref{KT1}). An element
$\vartheta '\in L$ is another Artin-Schreier generator of $L|K$ if and
only if
\begin{equation}
\vartheta ' = i \vartheta + c \mbox{\ \ \ with \ } c \in K \mbox{\ \
and\ \ }1\leq i \leq p-1.
\end{equation}
(cf.\ Lemma~\ref{allASg}). Consequently, using Lemma~\ref{aat} we
see that $\delta$ is an invariant of the extension $(L|K,v)$:
\begin{lemma}
The distance $\delta$ does not depend on the choice of the
Artin-Schreier generator $\vartheta$.
\end{lemma}
So we can call $\delta$ the \bfind{distance of the Artin-Schreier defect
extension $(L|K,v)$}. From Corollary~\ref{wpC} we know that
\[\delta\leq 0^-\;.\]

\pars
We will now distinguish two types of Artin-Schreier defect extensions.
We will call $(L|K,v)$ a \bfind{dependent Artin-Schreier defect extension}
if there exists an immediate purely inseparable extension $K(\eta)|K$ of
degree $p$ such that
\begin{equation}                            \label{atpi=AS}
\eta\;\sim_K\;\vartheta\;.
\end{equation}
Otherwise, we will speak of an \bfind{independent Artin-Schreier
defect extension}. For the definition and properties of the equivalence
relation ``$\sim_K$'', see Section~\ref{sectcad}. We will now show that
independent Artin-Schreier defect extensions are characterized by
idempotent distances $\delta$. See Lemma~\ref{charidemp} for a bunch of
different criteria which are all equivalent to ``$\delta$ is
idempotent''.

\begin{proposition}                    \label{dep}
In the situation as described above, the Artin-Schreier defect extension
$(L|K,v)$ is independent if and only if its distance $\delta$ is
idempotent:
\[\delta = p\delta\>.\]
\end{proposition}
\begin{proof}
Assume that $K(\eta)|K$ is purely inseparable of degree $p$, that is,
$\eta^p\in K\setminus K^p$. By definition, (\ref{atpi=AS}) is
equivalent to $v(\vartheta-\eta)>\delta$. Since $v(\vartheta^p-\eta^p)=
v(\vartheta-\eta)^p=pv(\vartheta-\eta)$, this in turn is equivalent to
\[v(\vartheta^p-\eta^p)\>>\> p\delta\;.\]
Here, the left hand side is equal to $v(\vartheta+a-\eta^p)=v(\vartheta-
(\eta^p-a))$ which is a value in $\Lambda^{\rm L}(\vartheta,K)$ and hence is
$\leq\delta$. Consequently, if (\ref{atpi=AS}) holds with $K(\eta)|K$ a
purely inseparable extension of degree $p$, then $p\delta<\delta$, that
is, $\delta$ is not idempotent.

For the converse, assume that $\delta$ is not idempotent. Since
$\delta\leq 0^-$, this implies that $p\delta<\delta$. Then there is
$c\in K$ such that $p\delta< v(\vartheta-c)\leq\delta$. Choose $\eta\in
\tilde{K}$ such that $\eta^p = a+c$. Then $v(\vartheta^p-\eta^p)=
v(\vartheta+a - \eta^p)= v(\vartheta-c)> p\delta$. Hence,
$v(\vartheta-\eta)>\delta$, and it follows that $\eta\sim_K\vartheta$.
Consequently, $\eta\notin K$, and we obtain that $K(\eta)|K$ is a purely
inseparable extension of degree $p$. Finally, we deduce from
Lemma~\ref{ueGp1} that this extension is immediate.
\end{proof}
\begin{corollary}                           \label{noinsnodep}
If $K$ admits no proper immediate purely inseparable extension, then
$K$ admits no dependent Artin-Schreier defect extension.        \QED
\end{corollary}

The converse of this corollary is not true: every
separable-algebraic\-ally closed non-trivially valued field $K$ of
characteristic $p > 0$ which is not algebraically closed is a
counterexample. Indeed, its value group is divisible and its residue
field is algebraically closed (see, e.g., [Ku2], Lemma~2.16) and hence,
the proper purely inseparable extension $\tilde{K}|K$ is immediate. But
a closer look shows that the irreversibility comes only from immediate
purely inseparable extensions which lie in the completion $K^c$ of $K$:

\begin{proposition}                           \label{ipie}
Assume that $K$ admits an immediate purely inseparable extension
$K(\eta)|K$ of degree $p$ such that $\eta\notin K^c$, and set
\[\varepsilon \;:=\; \dist(\eta,K)\;.\]
Then $K$ admits a dependent Artin-Schreier defect extension
$K(\vartheta)|K$. More precisely, given any $b \in K^\times$, then
\begin{equation}                            \label{condvb}
(p-1)vb+ v\eta\;>\; p\varepsilon
\end{equation}
if and only if there is an Artin-Schreier generator $\vartheta$ such
that $\vartheta^p-\vartheta=(\eta/b)^p$ and
\begin{eqnarray*}
\vartheta &\sim_K& \frac{\eta}{b}\\
v\vartheta &=& v\eta - vb\\
\dist(\vartheta,K) & = & \dist(\eta,K)-vb\>.
\end{eqnarray*}
All Artin-Schreier defect extensions obtained in this way are
dependent.
\end{proposition}
\begin{proof}
Let $\vartheta$ be a root of the polynomial
\begin{equation}
X^p - X - \left(\frac{\eta}{b}\right)^p \in K[X].\label{Zh}
\end{equation}
Assume that (\ref{condvb}) holds. Then we have
\begin{equation}                                         \label{eq5}
(p-1)vb + v\eta \;>\; p\varepsilon \;>\; p v\eta
\end{equation}
where the last inequality holds since $\varepsilon>v\eta$ by
Theorem~\ref{KT1}. This gives $vb>v\eta$, showing that
\[v\left(\frac{\eta}{b}\right)^p < \> 0\;.\]
Hence by Lemma~\ref{valASgen},
\begin{equation}                            \label{vtildth}
v\vartheta \>=\> v\,\frac{\eta}{b} \>=\> v\eta - vb\;.
\end{equation}
Putting $Y= bX$ we find that $b\vartheta$ is a root of the
polynomial
\begin{equation}                            \label{changepol}
Y^p - b^{p-1}Y - \eta^p \in K[Y]
\end{equation}
and thus satisfies
\[\eta^p +b^{p}\vartheta \;=\; \eta^p +b^{p-1}b\vartheta
\;=\; (b\vartheta)^p\;.\]
Let $c$ be an arbitrary element of $K$. By (\ref{vtildth}),
(\ref{condvb}) and the definition of $\varepsilon$,
\[vb^p\vartheta \;=\; pvb+v\eta-vb \;=\; (p-1)vb + v\eta\geq
p\varepsilon \;>\; pv(\eta-c) \;=\; v(\eta^p-c^p)\]
which yields, using the ultrametric triangle inequality,
\begin{eqnarray*}
v(\eta-c) & = & \frac{1}{p}v(\eta^p-c^p)\;=\;\frac{1}{p}\min
\{v(\eta^p-c^p)\,,\,vb^p\vartheta\}\\
& = & \frac{1}{p}v(\eta^p +b^p\vartheta -c^p) \;=\;
\frac{1}{p}v((b\vartheta)^p-c^p)\;=\;v(b\vartheta-c)\;.
\end{eqnarray*}
By Lemma~\ref{sim1} this implies that $b\vartheta\sim_K\eta$, which
by Lemma~\ref{aat} implies that
\[\vartheta \;\sim_K\; \frac{\eta}{b}\;.\]
From this, the assertion on the distance of $\vartheta$ follows by
virtue of Lemma~\ref{aat}, while the value $v\vartheta$ has already been
determined in (\ref{vtildth}). By Lemma~\ref{uniqextv}, the extension of
$v$ from $K$ to $K(\vartheta)$ is unique and $(K(\vartheta)|K,v)$ is an
Artin-Schreier defect extension. By definition, it is dependent.

For the converse, assume that (\ref{condvb}) does not hold, i.e.,
$(p-1)vb + v\eta\leq p\varepsilon$. If $v\left(\frac{\eta}{b}\right)^p>
0$, then by Lemma~\ref{valASgen}, $v\vartheta=v\left(\frac{\eta}{b}
\right)^p= pv\frac{\eta}{b} >v\frac{\eta}{b}$ and we cannot have
$\vartheta\sim_K \frac{\eta}{b}$. If $v\left(\frac{\eta}{b}\right)^p
\leq 0$, then again by Lemma~\ref{valASgen}, (\ref{vtildth}) holds, and
so we have
\[vb^p\vartheta \;=\; pvb+v\eta-vb \;=\; (p-1)vb +
v\eta \;\leq\; p\varepsilon\;.\]
Therefore, and since $\Lambda^{\rm L}(\eta,K)$ has no last element,
there is some $c\in K$ such that $vb^p\vartheta<pv(\eta-c)=
v(\eta^p-c^p)$. But then, by the ultrametric triangle inequality,
\[v(\eta-c) \;>\; \frac{1}{p}vb^p\vartheta \;=\;
\frac{1}{p}v(\eta^p +b^p\vartheta -c^p) \;=\; v(b\vartheta-c)\;,\]
which again shows that $\vartheta\sim_K \frac{\eta}{b}$ cannot be
true.
\end{proof}

The following proposition shows an even stronger independence property
than what is expressed in the definition:

\begin{proposition}                               \label{indep}
Let $(L|K,v)$ be an independent Artin-Schreier defect extension, and
take any element $\zeta\in L\setminus K$. Then there exists no
purely inseparable extension $K(\eta)|K$ such that $\zeta\sim_K\eta\,$.
In particular, it follows that
\begin{equation}                                     \label{eq7}
\dist(\zeta,K) = \dist (\zeta,K^{1/p^{\infty}})\;.
\end{equation}
\end{proposition}
\begin{proof}
Since $\zeta\in L\setminus K$, $[K(\zeta) :K]=p=[K(\vartheta):K]$ and
therefore, there is a polynomial $f\in K[X]$ of degree smaller than $p$
such that $\vartheta=f(\zeta)$. Suppose that there exists a purely
inseparable extension $K(\eta)|K$ such that $\zeta\sim_K\eta\,$. But
then by Lemma~\ref{fasimfb}, $\vartheta= f(\zeta) \sim_K f(\eta)\,$.
Since also $K(f(\eta))|K$ is a purely inseparable extension, this is
impossible since $(L|K,v)$ is assumed to be independent.

Equation (\ref{eq7}) is deduced as follows. If it does not hold, then
$\dist(\zeta,K) < \dist(\zeta, K^{1/p^{\infty}})$ in view of $K\subset
K^{1/p^{\infty}}$. But then by virtue of Lemma~\ref{sim2}, there would
exist some $\eta\in K^{1/p^{\infty}}$ such that $\zeta\sim_K\eta\,$,
which we have just shown not to be the case.
\end{proof}

%
%
\subsection{Deformation of Artin-Schreier defect extensions}
For the proof of Proposition~\ref{ipie}, we have transformed an
immediate purely inseparable extension into an immediate separable
extension. This was done by changing the minimal polynomial $Y^p-\eta^p$
to the minimal polynomial (\ref{changepol}) of $b\vartheta$
through addition of the summand $b^{p-1}Y$. The hypothesis on the value
of $b$ just means that it is large enough to guarantee that
$b\vartheta \sim_K\eta$. For this hypothesis, it is necessary that
$\eta$ is not contained in the completion of $K$. On the other hand, an
immediate purely inseparable extension with a generator $\eta$ in the
completion of $K$ cannot be transformed into any immediate separable
extension with a generator $\vartheta$ such that $\vartheta\sim_K\eta$.
Indeed, if $\eta\in K^c$ and $\eta\sim_K\eta'$, then $v(\eta-\eta')>
\widetilde{vK}$, that is, $\eta=\eta'$. Moreover, every henselian field
$K$ is separable-algebraically closed in its completion (cf.\ [W],
Theorem~32.19).

The general idea of the transformation of the minimal polynomial can be
expressed as follows: if $y\notin K^c$ is a root of the polynomial $f\in
K[X]$, then for a given polynomial $g\in K[X]$, a root $z$ of $g$ will
satisfy $y\sim_K z$ as soon as the coefficients of the polynomial $f-g$
have large enough values. This follows in general from the principle of
Continuity of Roots. But we wanted to give a self-contained proof for
our special case, because it is particularly simple and explicit and
leads to the following deformation theory.

For any fixed $a\in K$, we consider the following family of polynomials
defined over $K$:
\begin{equation}
f_{a,b}(Y)\;:=\;Y^p-b^{p-1}Y-a\,,\qquad b\in K^\times\;.
\end{equation}
This family can be viewed as a deformation of the polynomial
$Y^p-a$, with this polynomial as its limit for $vb\rightarrow\infty$:
\begin{eqnarray*}
Y^p-b^{p-1}Y-a & \longrightarrow & Y^p-a\\
vb & \longrightarrow & \infty \;.
\end{eqnarray*}
But it is not necessarily true that the ramification theoretical
properties are preserved in the limit, as Example~\ref{examp1}
in Section~\ref{sectexamp} will show.

Associated with this family through the transformation $Y=bX$ is the
family
\begin{equation}
g_{a,b}(X)\;:=\;X^p-X-\frac{a}{b^p}\,,\qquad b\in K^\times\;,
\end{equation}
where $\vartheta_{a,b}$ is a root of $g_{a,b}$ if and only if
$b\vartheta_{a,b}$ is a root of $f_{a,b}\,$.

We summarize the properties of these families in the following theorem:
\begin{theorem}                             \label{deform}
\n
a) \ If $pvb\geq va$, then the polynomial $g_{a,b}(X)$ induces a
Artin-Schreier extension for which equality holds in the fundamental
inequality (\ref{fundineq}); if $pvb> va$, then this extension lies in
the henselization of $K$.
\n
b) \ Suppose that the polynomial $Y^p-a$ induces an immediate extension
which does not lie in the completion of $K$. Then for each $b\in
K^\times$ of large enough value, the polynomial $g_{a,b}(X)$ induces a
dependent Artin-Schreier defect extension; every root $b\vartheta_{a,b}$
of $f_{a,b}(X)$ will then satisfy
\[b\vartheta_{a,b}\sim_K a^{1/p}\;.\]
``Large enough value'' means that
\begin{equation}                            \label{condb}
(p-1)vb + \frac{va}{p} \> >\> p\,\dist(a^{1/p},K)\;.
\end{equation}
If this condition is violated, then $b\vartheta_{a,b}\sim_K
a^{1/p}$ does not hold.
\n
c) \ Suppose that a root $\vartheta_{a,1}$ of the polynomial $f_{a,1}(X)
=X^p-X-a$ satisfies
\begin{equation}                            \label{condc}
v\vartheta_{a,1}\> >\>p\,\dist(\vartheta_{a,1},K)\;.
\end{equation}
Then the polynomial $X^p-a$ induces an immediate extension which does
not lie in the completion, and for every $b$ in the valuation ring
${\cal O}$ of $K$ and every root $\vartheta_{a,b}$ of $g_{a,b}$,
$K(\vartheta_{a,b})|K$ is a dependent Artin-Schreier defect extension
with $b\vartheta_{a,b} \sim_K a^{1/p}$. If condition (\ref{condc}) is
violated, then $\vartheta_{a,1} \sim_K a^{1/p}$ does not hold.
\end{theorem}
\begin{proof}
a): \ Both assertions follow from Lemma~\ref{classASE}.
\sn
b): \ All assertions follow from Proposition~\ref{ipie} where
$\eta=a^{1/p}$.
\sn
c): \ Assume that condition (\ref{condc}) holds. Then it follows from
the second part of the proof of Proposition~\ref{dep}, where we set
$c=0$ and $\vartheta= \vartheta_{a,1}\,$, that the polynomial $X^p-a$
induces an immediate extension which does not lie in the completion, and
that $\vartheta_{a,1} \sim_K a^{1/p}$. The latter implies that
$v\vartheta_{a,1} = va^{1/p}=\frac{va}{p}$ and that
$\dist(\vartheta_{a,1},K)= \dist(a^{1/p},K)$; hence,
it implies that (\ref{condc}) is equivalent to
\begin{equation}                            \label{cond1}
\frac{va}{p}\> >\>p\,\dist(a^{1/p},K)\;.
\end{equation}
Consequently, (\ref{condb}) will hold for every
$b\in {\cal O}$, so it follows from part b) that for every root
$\vartheta_{a,b}$ of $g_{a,b}$, $K(\vartheta_{a,b})|K$ is a dependent
Artin-Schreier defect extension with $b\vartheta_{a,b} \sim_K a^{1/p}$.

The last assertion of part c) is seen as follows. We have shown that if
$\vartheta_{a,1} \sim_K a^{1/p}$ holds, then (\ref{condc}) and
(\ref{cond1}) are equivalent. But if (\ref{cond1}) is violated, then by
part b), $\vartheta_{a,1} \sim_K a^{1/p}$ cannot hold.
\end{proof}
\n
Note that
\begin{equation}                            \label{pd=d^p}
p\,\dist(a^{1/p},K)\>=\>\dist(a,K^p)\;.
\end{equation}

\parm
A deformation which at first sight seems to be different from the above
has been used by B.~Teissier in [T]. Starting from the Artin-Schreier
polynomial $X^p-X-a$, we set $X=aY$ and then divide the polynomial by
$a^p$, which leads to the polynomial
\[Y^p-a^{1-p}Y-a^{1-p}\>=\>Y^p-a^{1-p}(1+Y)\;.\]
Hence, $\vartheta^p-\vartheta=a$ if and only if for
$\tilde\vartheta=\vartheta/a$,
\begin{equation}                            \label{eq1unit}
\tilde\vartheta^p-a^{1-p}(1+\tilde\vartheta)=0\;.
\end{equation}
We assume that $va<0$. Then $\vartheta^p-\vartheta=a$ implies that
$va=pv\vartheta<v\vartheta$ and therefore,
\[v\tilde\vartheta\>=\>v\vartheta-va\>>\>0\;.\]
That is, $1+\tilde\vartheta$ is a $1$-unit in ${\cal O}$.
Reducing this $1$-unit to $1$ deforms equation (\ref{eq1unit}) to
\[\ovl{\vartheta}^p-a^{1-p}\>=\>0\;,\]
viewed as an equation in an associated graded ring. In fact, we have
reduced equation (\ref{eq1unit}) modulo the ${\cal O}$-ideal
\[a^{1-p}\tilde\vartheta{\cal O}\>=\>a^{-p}\vartheta {\cal O} \>=\>
\vartheta^{1-p^2} {\cal O}\;.\]
Analyzing the above transformation, one sees that its advantage is that
it leeds to equations with integral coefficients. However, if we
multiply the polynomial $Y^p-a^{1-p}$ by $a^p$ and then set $X=aY$, we
obtain the polynomial $X^p-a$. So we have just replaced the polynomial
$X^p-X-a$ by $X^p-a$. From Theorem~\ref{deform} together with
(\ref{pd=d^p}) we see that this procedure preserves the valuation
theoretical behaviour of the associated roots if and only if
\[va \> >\> p\,\dist(a,K^p)\;.\]

%
%
\subsection{Fields without dependent Artin-Schreier defect extensions}
\label{sectfwdASe}
If $K$ admits any immediate purely inseparable extension that does not
lie in the completion $K^c$ of $K$, then $K$ satisfies the hypothesis of
Proposition~\ref{ipie}. To show this, suppose that $\tilde\eta\in
K^{1/p^{\infty}} \setminus K^c$ such that $K(\tilde\eta)|K$ is an
immediate extension. We may assume that $\tilde\eta^p\in K^c$
(otherwise, we replace $\tilde \eta$ by a suitable $p^{\nu}$-th power).
Since $\tilde\eta\notin K^c$, we have that $\Lambda^{\rm L}
(\tilde\eta,K)$ is bounded from above in $vK$ and $\Lambda^{\rm L}
(\tilde\eta^p,K^p)= p\Lambda^{\rm L}(\tilde\eta,K)$ is bounded from
above in $vK^p=pvK$. On the other hand, since $\tilde{\eta}^p\in K^c$,
there is some $b\in K$ such that $v(\tilde\eta^p-b)>\Lambda^{\rm L}
(\tilde\eta^p,K^p)$. We choose $\eta\in K^{1/p}$ such that $\eta^p=b$.
Then $v(\tilde\eta-\eta)= \frac{1}{p}v(\tilde\eta^p-b)> \Lambda^{\rm
L}(\tilde\eta,K)$, that is,
\[\eta\;\sim_K\;\tilde\eta\;.\]
By Lemma~\ref{ueGp1}, this shows that $K(\eta)|K$ is an immediate
extension; since $\Lambda^{\rm L}(\eta,K)=\Lambda^{\rm L}
(\tilde\eta,K)\ne vK$, it is not contained in $K^c$. We may now apply
Proposition~\ref{ipie} to obtain:
\begin{corollary}                           \label{+4.5}
Assume that $K$ does not admit any dependent Artin-Schreier defect
extension. Then every immediate purely inseparable extension lies in the
completion of $K$.                                      \QED
\end{corollary}

\begin{lemma}                               \label{ASdown}
%
If $K$ is Artin-Schreier closed, then so is $K^c$. If $K$ admits no
dependent (or no independent) Artin-Schreier defect extension, then the
same holds for $K^c$.
\end{lemma}
\begin{proof}
Assume that $K^c(\vartheta)|K$ is an Artin-Schreier extension generated
by a root $\vartheta$ of the polynomial $X^p-X-a$ over $K^c$. Since
$\vartheta\notin K^c$, we have that $\dist(\vartheta, K^c)<\infty$.
Since $a\in K^c$, we may choose an element $\tilde{a}\in K$ such that
$v(a-\tilde{a}) >\dist(\vartheta,K^c)$ with $v(a-\tilde{a}) \geq 0$. Let
$\tilde{\vartheta}$ be a root of the polynomial $X^p-X-\tilde{a}\in
K[X]$. By Lemma~\ref{valASgen}, the root $\vartheta-\tilde{\vartheta}$
of the polynomial $X^p-X-(a-\tilde{a})$ has value $v(\vartheta-
\tilde{\vartheta})=v(a-\tilde{a})>\dist( \vartheta,K^c)\geq
\dist(\vartheta,K)$. Thus, $\dist(\tilde{\vartheta},K)=\dist(\vartheta,K)
\leq\dist(\vartheta,K^c) <\infty$, which shows that
$K(\tilde{\vartheta})|K$ is non-trivial and hence an Artin-Schreier
extension. This proves the first assertion of our lemma.

Now assume that $(K^c(\vartheta)|K,v)$ is an Artin-Schreier defect
extension. By Corollary~\ref{wpC} we have that $\dist(\vartheta,K^c)\leq
0^-$. With $\tilde{\vartheta}$ as before, we obtain that
$\dist(\tilde{\vartheta},K)=\dist(\vartheta,K)\leq 0^-$. By
Lemma~\ref{uniqextv}, this shows that also $(K(\tilde{\vartheta})|K,v)$
is an Artin-Schreier defect extension. The equality of the distances
shows that $K^c(\vartheta)|K^c$ is independent if and only if
$K(\tilde{\vartheta})|K$ is.
\end{proof}

An immediate consequence of this lemma and the preceding corollary is:

\begin{corollary}                           \label{nodepincomp}
If $K$ does not admit any dependent Artin-Schreier defect extension,
then $K^c$ does not admit any proper immediate purely inseparable
extension. In particular, this holds if $K$ is separable-algebra\-ically
maximal.                                                \QED
\end{corollary}

\pars
We can now give the
\sn
{\bf Proof of Theorem~\ref{AScphic}:} \
Every Artin-Schreier closed non-trivially valued field $K$ of
characteristic $p>0$ has $p$-divisible value group and perfect residue
field (cf.\ Corollary~2.17 of [Ku2]). Therefore, every purely
inseparable extension of $K$ is immediate. Hence by the last corollary,
the perfect hull of $K$ lies in the completion of $K$, i.e., $K$ lies
dense in its perfect hull.

An alternative proof of this fact can be given in the following way. We
represent the extension $K^{1/p^{\infty}}|K$ as an infinite tower of
purely inseparable extensions $K_{\mu+1}|K_\mu$ ($\mu<\nu$ where $\nu$
is some ordinal). Then we only have to show that $(K_{\mu+1},v)$ lies in
$(K_\mu,v)^c$ for every $\mu<\nu$. In view of Proposition~\ref{ipie}, it
suffices to show that $K_\mu$ is Artin-Schreier closed. But this holds
by Lemma~\ref{AScpieASc}.

Since $K^c$ has the same value group and the same residue field as $K$,
also every purely inseparable extension of $K^c$ is immediate. By the
preceding corollary, this yields that $K^c$ must be perfect.     \QED

%
%
\subsection{Persistence results}
Another property of independent Artin-Schreier defect extensions is
their persistence in maximal immediate extensions, in the following
sense:

\begin{lemma}                                     \label{l2}
If $K$ admits an independent Artin-Schreier defect extension
$(K(\vartheta)|K,v)$ with Artin-Schreier generator $\vartheta$ of
distance $\delta = 0^-$, then every algebraically maximal immediate
extension (and in particular, every maximal immediate extension) $M$ of
$K$ contains also an independent Artin-Schreier defect extension of $K$
with an Artin-Schreier generator $\tilde\vartheta$ of distance $0^-$
such that $\tilde\vartheta\sim_K\vartheta$.
\end{lemma}
\begin{proof}
If $\vartheta\in M$, there is nothing to show. Assume that $\vartheta
\notin M$. Then $M(\vartheta)|M$ is also an Artin-Schreier extension
with Artin-Schreier generator $\vartheta$. Since $M$ is algebraically
maximal, Corollary~\ref{exoptappr} shows that there exists an element
$u \in M$ satisfying
\[v(\vartheta - u) \;\geq\; \Lambda^{\rm L}(\vartheta,M)\;.\]
On the other hand, $K \subseteq M$ implies
\[\Lambda^{\rm L}(\vartheta,K)\subseteq \Lambda^{\rm L}(\vartheta,M)\;.\]
Since $vM=vK$, this shows that $v(\vartheta - u)\geq 0$. We put
\[a_u\;:=\;\wp (\vartheta - u)\;=\;\wp (\vartheta)-\wp (u)\in M\]
and note that $va_u\geq 0$. Since $M|K$ is immediate, there exists
$b \in K$ such that
\[v(a_u - b)>v(a_u)\geq 0\]
and $vb=va_u\geq 0$. Consequently, the polynomial $X^p-X - (a_u-b) \in
M[X]$ admits a root $\vartheta'$ in the henselian field $M$. But then,
\[\tilde{\vartheta}:=\vartheta'+u\in M\]
is a root of the polynomial $X^p-X -(\wp (\vartheta)-b) \in K[X]$. We
compute:
\[\wp(\vartheta-\tilde{\vartheta}) = \wp(\vartheta)-\wp(\vartheta'+u)
= \wp (\vartheta) - (\wp (\vartheta)-b) = b\;.\]
This shows $v(\vartheta-\tilde{\vartheta})\geq 0$, whence
$\tilde\vartheta\sim_K\vartheta$. In particular, this shows that
$\tilde\vartheta \notin K$ so that $K(\tilde\vartheta)| K$ is
non-trivial and hence an
Artin-Schreier extension. By Lemma~\ref{uniqextv}, the extension of $v$
from $K$ to $K(\tilde\vartheta)$ is unique and $K(\tilde\vartheta)|K$ is
an Artin-Schreier defect extension. Finally, $\tilde\vartheta\sim_K
\vartheta$ implies that $\dist(\tilde\vartheta,K)=\dist(\vartheta,K)
=0^-$ (Lemma~\ref{sim1}) and therefore, $K(\tilde\vartheta)|K$ is an
independent Artin-Schreier defect extension.
\end{proof}

From this lemma, we deduce the following:
\begin{corollary}                           \label{sacinmax}
If there exists a maximal immediate extension in which $K$ is
separable-algebraically closed, then $K$ admits no independent
Artin-Schreier defect extension of distance $0^-$.
\end{corollary}

\pars
We will now consider independent Artin-Schreier defect extensions
$(K(\vartheta)|K,v)$ with Artin-Schreier generator $\vartheta$ of
distance $\delta< 0^-$. In this case, Lemma~\ref{charidemp} shows that
$\delta= H^-$ for some non-trivial convex subgroup $H$ of
$\widetilde{vK}$. This means that $v(\vartheta-K)= \Lambda^{\rm L}
(\vartheta,K)$ is cofinal in $(\widetilde{vK})^{<0}\setminus H$. We
denote by $v_\delta$ the coarsening of $v$ on $\tilde K$ with respect to
$H$. Then $v_\delta(\vartheta-K)$ is cofinal in $(\widetilde{vK})^{<0}
/H= (\widetilde{v_\delta K})^{<0}$. Thus, $v_\delta(\vartheta-K)$ has no
maximal element. Since the extension of $v$ from $K$ to $K(\vartheta)$
is unique, the same must hold for $v_\delta$; cf.\ the proof of
Lemma~\ref{defc}. Now Lemma~\ref{ueGp1} shows that also $(K(\vartheta)|
K,v_\delta)$ is an immediate Artin-Schreier extension. As its distance
is $0^-$, it is covered by the case treated in Lemma~\ref{l2}. From
this, we obtain:

\begin{lemma}                            \label{l3}
Assume that for every coarsening $w$ of $v$ (including $v$ itself), there
exists a maximal immediate extension $(M_w,w)$ of $(K,w)$ such that $K$ is
separable-algebraically closed in $M_w$. Then $K$ admits no independent
Artin-Schreier defect extensions.                             \QED
\end{lemma}

The condition of Lemma~\ref{l3} is preserved under finite defectless
extensions:

\begin{lemma}                            \label{l7}
Assume that for every coarsening $w$ of $v$ (including $v$ itself),
$K_0$ admits a maximal immediate extension $(N_w|K_0,w)$ such that $K_0$
is relatively algebraically closed (or separable-algebraically
closed) in $N_w$. If the extension $(K|K_0,v)$ is finite and defectless,
then for every coarsening $w$ of~$v$ (including $v$ itself), $(M_w,w) =
(N_w.K,w)$ is a maximal immediate extension of $(K,w)$ such that $K$ is
relatively algebraically closed (or separable-algebraically closed,
respectively) in $M_w$.
\end{lemma}
\begin{proof}
Since $(K|K_0,v)$ is defectless by hypothesis, the same is true for the
extension $(K|K_0,w)$ by Lemma~\ref{defc}. We note that $(K_0,w)$ is
henselian since it is assumed to be separable-algebraically closed in
the henselian field $(N_w,w)$. So we may apply Lemma~\ref{ivd}: since
$(N_w|K_0,w)$ is immediate and $(K|K_0,w)$ is defectless,
$(N_w.K|K,w)$ is immediate and $N_w$ is linearly disjoint from $K$
over $K_0$. The latter shows that $K$ is relatively algebraically closed
(or separable-algebraically closed, respectively) in $N_w.K$. On the
other hand, $(M_w,w)=(N_w.K,w)$ is a maximal field, being a finite
extension of a maximal field.
\end{proof}

\begin{proposition}                            \label{l6}
If $K_0$ is a separable-algebraically maximal field and $K|K_0$ is a
finite defectless extension, then $K$ admits no independent
Artin-Schreier defect extensions.
\end{proposition}
\begin{proof}
Let $w$ be any coarsening of $v$. Since $(K_0,v)$ is
separable-algebra\-ically maximal, the same is true for $(K_0,w)$ since
every finite separable immediate extension of $(K_0,w)$ would also be
immediate for the finer valuation $v$. Now let $(N_w,w)$ be a maximal
immediate extension of $(K_0,w)$. Since $(K_0,w)$ is
separable-algebraically maximal, it is separable-algebraically closed in
$N_w$. Hence, $K_0$ satisfies the condition of Lemma~\ref{l7}. So
our proposition is a consequence of Lemma~\ref{l7} together with
Lemma~\ref{l3}.
\end{proof}

%
%
\subsection{Generalization of Lemma~\ref{delon} and proof of
Theorem~\ref{si=d}}                         \label{sectth1}
For the generalization of Lemma~\ref{delon}, we will need the following
result:
\begin{lemma}                                        \label{l4}
Let $K \subset K_1 \subset K_2$ be extensions of valued fields of
characteristic $p > 0$ such that $K_1|K$ is finite and purely
inseparable and $K_2|K_1$ is an independent Artin-Schreier defect
extension. Then there exists an Artin-Schreier extension $L|K$ such that
$K_2 = K_1.L$, and every such extension $L|K$ is an independent
Artin-Schreier defect extension.
\end{lemma}
\begin{proof}
Let $\tilde{\vartheta}$ be an Artin-Schreier generator of $K_2| K_1$
and choose $\nu \geq 1$ such that
\[K_1^{p^{\nu}}\! \subseteq K\;.\]
Then
\[\wp (\tilde{\vartheta}^{p^\nu}) = (\wp (\tilde{\vartheta}))^{p^\nu}
\in K\;,\]
hence
\[K(\tilde{\vartheta}^{p^\nu})| K\]
is an Artin-Schreier extension: it is non-trivial since
$K(\tilde{\vartheta})|K$ is not purely inseparable. Comparing
degrees, we see that $K_2 = K_1 (\tilde{\vartheta}^{p^\nu})=
K_1.K(\tilde{\vartheta}^{p^\nu})$.

Now let $L | K$ be any such Artin-Schreier extension. Let $\vartheta$
be an Artin-Schreier generator of $L| K$ and hence of $K_2| K_1$ too.
Using $\vartheta^p = \vartheta + a$ with $a\in K$, we compute
\begin{equation}                            \label{th^p^n}
\vartheta^{p^\nu} = \vartheta + a' \mbox{\ \ where\ \ }
a' = a + \ldots + a^{p^{\nu - 1}} \in K\;.
\end{equation}
Hence,
\[\dist(\vartheta^{p^\nu},K_1)=\dist(\vartheta,K_1)\>.\]
Further,
\[\delta := \dist(\vartheta,K_1) =p^\nu\delta
= \dist(\vartheta^{p^\nu},K_1^{p^\nu})\]
since $\delta$ is idempotent by hypothesis;
\[\dist(\vartheta^{p^\nu},K_1^{p^\nu}) \leq
\dist(\vartheta^{p^\nu},K)
\leq \dist(\vartheta^{p^\nu},K_1)\]
because $K_1^{p^\nu}\! \subseteq K \subset K_1$.
Putting these three equations together, we find that equality holds
everywhere. In particular,
\[\dist(\vartheta,K_1)=\dist(\vartheta^{p^\nu},K)
=\dist(\vartheta,K)\;,\]
where the second equality again holds because of (\ref{th^p^n}). This
shows that $\Lambda^{\rm L}(\vartheta,K)$ is cofinal in $\Lambda^{\rm L}
(\vartheta, K_1)$. Since $K_1(\vartheta)|K_1$ is immediate, we know from
Theorem~\ref{KT1} that $\Lambda^{\rm L}(\vartheta,K_1)=v(\vartheta-
K_1)$ has no maximal element. Now we have that $\Lambda^{\rm L}
(\vartheta,K) \subseteq v(\vartheta-K)\subseteq v(\vartheta-K_1)$ and
that $\Lambda^{\rm L}(\vartheta,K)$ is cofinal in $v(\vartheta-K_1)$;
this yields that $v(\vartheta-K)$ is cofinal in $v(\vartheta-K_1)$ and
thus has no maximal element. Now Lemma~\ref{ueGp1} shows that
$K(\vartheta)|K$ is immediate. Since $\dist(\vartheta,K)= \dist
(\vartheta,K_1)$ is idempotent, $K(\vartheta)|K$ is independent.
\end{proof}

\begin{lemma}                                 \label{l5}
Every finite extension of an inseparably defectless field
of characteristic $p > 0$ is again an inseparably defectless field.
\end{lemma}
\begin{proof}
From Corollary~\ref{ed} it follows that every finite purely inseparable
extension of an inseparably defectless field is again an inseparably
defectless field. Thus it remains to show the lemma in the case of a
finite separable extension $L$ of an inseparably defectless field $K$.
We fix an extension of $v$ to $K\sep$ and consider the ramification
fields $K^r$ and $L^r$ of $K$ and $L$ with respect to that extension. By
Proposition~\ref{dlta}, we know that $K$ is inseparably defectless if
and only if $K^r$ is inseparably defectless, and the same holds for $L$
and $L^r$. By Lemma~\ref{rf}, we have $L^r = L.K^r$, and therefore $L^r|
K^r$ is a finite separable extension. The same proposition shows that
$K\sep| K^r$ is a $p$-extension, so $L^r| K^r$ is a tower of
Artin-Schreier extensions (cf.\ Lemma~\ref{tow}). Hence, replacing $K$
and $L$ by their ramification fields, we may assume from the start that
they are henselian and that $L|K$ is a tower of Artin-Schreier
extensions. Now it suffices to prove that $L$ is inseparably defectless
under the additional assumption that $L|K$ itself is an Artin-Schreier
extension since then, our assertion will follow by induction. Since
$L^{1/p^{\infty}} = L.K^{1/p^{\infty}}$, it
suffices to show for every finite purely inseparable extension $K_1| K$
(which itself is defectless by hypothesis), that $K_2 = K_1.L$ is a
defectless extension of $L$. This follows immediately if $K_2| K_1$ and
thus $K_2| K$ are defectless. Now assume that $K_2| K_1$ is immediate.
Note that $K_1$ is an inseparably defectless field, being a finite
purely inseparable extension of the inseparably defectless field $K$. In
particular, this yields that $K_1$ admits no immediate purely
inseparable extension and hence by virtue of Proposition~\ref{dep}, no
dependent Artin-Schreier defect extension. The immediate Artin-Schreier
extension $K_2| K_1$ is thus independent. An application of
Lemma~\ref{l4} now shows that $L|K$ is immediate. But then, it follows
already from Corollary~\ref{ivd} that $K_2|L$ is defectless. Hence we
have proved that $L$ is an inseparably defectless field.
\end{proof}
\n
In both of the preceding lemmas, the finiteness conditions cannot be
dropped, as Examples~\ref{examp1} and~\ref{examp4} in the next
section will show.

\parb
We are now able to give the
\sn
{\bf Proof of Theorem~\ref{si=d}: } \n
Assume that the valued field $K$ of characteristic $p>0$ is
separable-algebraically maximal and inseparably defectless.
We note that $K$ is henselian since it is separable-algebraically
maximal. Let $(L|K,v)$ be a finite extension. We want to show that it is
defectless. Since any subextension of a defectless extension is
defectless too, we may assume w.l.o.g. that $L|K$ is normal. Hence there
exists an intermediate field $K_1$ such that $L| K_1$ is separable and
$K_1 | K$ is purely inseparable. By hypothesis, we know that $K_1 | K$
is defectless. It remains to prove that $L| K_1$ is defectless.

Using Lemma~\ref{tow}, choose a finite tame extension $N$ of $K_1$ such
that $L.N|N$ is a tower of Artin-Schreier extensions.
By Proposition~\ref{dlta}, $L|K_1$ is defectless if and only if $L.N|N$
is defectless. Since $K_1|K$ is defectless and $N|K_1$ is tame and
hence defectless, both extensions being finite, $N|K$ is finite and
defectless. Using Lemma~\ref{l5} we conclude that $N$ is inseparably
defectless too and therefore does not admit immediate purely inseparable
extensions. By Corollary~\ref{noinsnodep}, this shows that every
immediate Artin-Schreier extension of the henselian field $N$ must be
independent. Moreover, from Proposition~\ref{l6} we infer that $N$ does
not admit independent Artin-Schreier defect extensions. Consequently,
given an Artin-Schreier extension $L'|N$ contained in $L.N|N$, this
extension must be defectless. In view of Lemma~\ref{l5} and
Proposition~\ref{l6}, $L'$ will again be inseparably defectless and will
not admit any independent Artin-Schreier defect extension. By induction,
we conclude that all Artin-Schreier extensions in the tower $L.N|N$
are defectless, hence $L.N|N$ and thus $L|K_1$ and $L|K$ are defectless,
as asserted.

\pars
Conversely, every defectless field is immediately seen to be
separable-algebraically maximal and inseparably defectless.
                                                                \QED

%
%
\subsection{Examples}                       \label{sectexamp}
\begin{example}                             \label{examp1}
{\bf (for an independent Artin-Schreier defect extension with distance
$0^{-}$):} \ Let $k$ be an algebraically closed field of characteristic
$p>0$, and $K=k(t)^{1/p^\infty}$ the perfect hull of the rational
function field $k(t)$. Further, let $v=v_t$ be the unique extension of
the $t$-adic valuation from $k(t)$ to $K$; we write $vt=1$. Note that
$vK$ is $p$-divisible and $Kv=k$ is algebraically closed.

We consider the Artin Schreier extension $L_0=k(t,\vartheta)$ of
$k(t)$ generated by a root $\vartheta$ of the polynomial
\[X^p\,-\,X\,-\,\frac{1}{t}\;.\]
As $v\vartheta=-1/p\notin \Z=vk(t)$, we see that
$[L_0:k(t)]=p=(vL_0:vk(t))$. Thus, the extension of $v$ from $k(t)$ to
$L_0$ is unique. Further, the extension of $v$ from $L_0$ to its perfect
hull is unique. But the latter is equal to $L_0.K$, so we find that the
extension of $v$ from $K$ to $L:=L_0.K$ is unique. On the other hand,
the extension $L|K$ is immediate since $vK$ is $p$-divisible and $Kv=k$
is algebraically closed. Therefore, $L|K$ is an Artin-Schreier defect
extension. Since $K$ is perfect, it is independent by definition.

For
\[a_n\;:=\;\sum_{i=1}^{n} \frac{1}{t^{p^{-i}}}\]
we have
\[a_n^p\,-\,a_n\;=\; \frac{1}{t}\,-\,\frac{1}{t^{p^{-n}}}\;,\]
whence
\[(\vartheta-a_n)^p \,-\,(\vartheta-a_n)\;=\;\vartheta^p-\vartheta\,-\,
(a_n^p-a_n)\;=\;\frac{1}{t}\,-\,\left(\frac{1}{t}\,-\,\frac{1}{t^{p^{-n}}}
\right) \;=\; \frac{1}{t^{p^{-n}}}\;.\]
By Lemma~\ref{valASgen} this yields
\[v(\vartheta-a_n)\;=\; \frac{1}{p}\>v\,\frac{1}{t^{p^{-n}}}\;=\;
-\frac{1}{p^{n+1}}\;.\]
Since this increases with $n$, we see that $(a_n)_{n\in\N}$ is a pseudo
Cauchy sequence with limit $\vartheta$. By Corollary~\ref{wpC},
$\dist(\vartheta,K)\leq 0^-$. On the other hand, the values
$v(\vartheta-a_n)$ are cofinal in $\widetilde{vK}^{<0}$. Therefore,
\[\dist(\vartheta,K)\;=\;0^{-}\;.\]
\pars
This example shows that the condition in Lemma~\ref{l4} that $K_1|K$ be
finite cannot be dropped. Indeed, it is known that $(k(t),v_t)$ is a
defectless field (for instance, this is a consequence of the Generalized
Stability Theorem, cf.\ [Ku4]). So it does not admit any Artin-Schreier
defect extension. But the infinite extension $K$ of $k(t)$ admits an
independent Artin-Schreier defect extension.

\pars
The example also shows that ramification theoretical properties of a
polynomial are not necessarily preserved in the limit. As above, one
shows that for every $n\in\N$, a root of the polynomial
\[X^p-X-\frac{1}{t^{np+1}}\]
generates a non-trivial immediate extension of $K$. The same is true for
a root of the polynomial
\[Y^p-t^{n(p-1)}Y-\frac{1}{t}\;.\]
Under $n\rightarrow\infty$ (which implies
$vt^{n(p-1)}\rightarrow\infty$), the limit of this polynomial is
\[Y^p-\frac{1}{t}\;.\]
But this polynomial does not induce a non-trivial extension of $K$ since
$K$ is perfect.
\end{example}

This example works even for non-algebraically closed fields $k$. In
[Ku2] we presented it with $k=\F_p$. See also [Ku5].

\begin{example}                              \label{examp2}
{\bf (for an independent Artin-Schreier defect extension with distance
smaller than $0^{-}$):} \ In the previous example, we may choose $k$
such that it admits a non-trivial valuation $\ovl{v}$. Now we consider
the valuation $v':=v\circ \ovl{v}$ on $L$. As $(L|K,v)$ is immediate and
$Lv=k=Kv$, it follows that also $(L|K,v')$ is immediate. The value group
$\ovl{v}k$ is canonically isomorphic to a non-trivial convex subgroup
$H$ of $v'L$ (such that $v'L/H\isom vL$). If there would exist some
$c\in K$ and an element $\beta\in H$ such that $v'(\vartheta-c)\geq
\beta$, then $v(\vartheta-c)\geq 0$ which is impossible. On the other
hand, the values $v'(\vartheta-a_n)$ are cofinal in $\{\alpha\in
\widetilde{v'K}\mid \alpha<H\}$ since the values $v(\vartheta-a_n)$ are
cofinal in $vK^{<0}$. This shows that the distance $\dist(\vartheta,K)$
with respect to $v'$ is the cut
\[H^{-}\;=\; (\{\alpha\in \widetilde{v'K}\mid \alpha<H\}\,,\,\{\alpha\in
\widetilde{v'K}\mid\exists \beta\in H:\,\beta\leq\alpha\})\]
which is smaller than $0^{-}$ since $H$ is non-trivial.
\end{example}

\begin{example}                             \label{examp3}
{\bf (for a dependent Artin-Schreier defect extension):} \
With $k(t)$ as before, we take $K_0$ to be the separable-algebraic
closure of $k(t)$, with any extension $v_t$ of the $t$-adic valuation of
$k(t)$. Being separable-algebraically closed, $K_0$ does in particular
not admit any Artin-Schreier extension. But we can build a field
admitting a dependent Artin-Schreier defect extension by taking $K=K_0(x)$
and endowing it with the (unique) extension $v$ of $v_t$ such that
$vx>vK_0$. (This means that $K$ has the $x$-adic valuation $v_x$ with
residue field $K_0$, and $v=v_x\circ v_t$ is the composition of $v_x$
with $v_t\,$.) We take any $\eta\in \tilde{K_0}\setminus K_0$. Since
$\eta$ lies in the completion of $(K_0,v)$ by Theorem~\ref{AScphic}, we
have $\Lambda^{\rm L}(\eta,K_0)=v_tK_0=vK_0\,$. It follows that
$\Lambda^{\rm L}(\eta,K)$ is the least initial segment of $vK$ containing
$vK_0$. That is, the cut $\dist(\eta,K)$ is the cut $(vK_0)^+$ induced
in $\widetilde{vK}$ by the upper edge of the convex subgroup $vK_0$ of
$vK$. In particular, $\eta$ does not lie in the completion of $(K,v)$.
Now Proposition~\ref{ipie} shows that $K$ admits a dependent
Artin-Schreier defect extension. According to this proposition, it can
for instance be generated by a root $\vartheta$ of the polynomial
$X^p-X-(\eta/x)^p$, as $vx>\dist(\eta,K)= p\,\dist(\eta,K)$. Then
$\dist(\vartheta,K)=
\dist(\eta,K)-vx=(vK_0)^+ -vx=(-vx+vK_0)^+$ is the cut induced by the
upper edge of the coset $-vx+vK_0$ in $\widetilde{vK}$. Note that in
$vK$, which is the lexicographic product $\Z vx\times vK_0\,$, the cut
$(-vx+vK_0)^+$ is equal to the cut $vK_0^-$ induced by the lower edge of
the convex subgroup $vK_0$ of $vK$. Nevertheless, the cut
$\dist(\vartheta,K)$ in $\widetilde{vK}$ is {\it not} equal to $H^-$
or $H^+$ for any convex subgroup $H$ of $vK$ or of $\widetilde{vK}$
(cf.\ Example~\ref{exampcut} in Section~\ref{sectcad}).
\end{example}

Enlarging the rank of the valuation in order to obtain a dependent
Artin-Schreier defect extension may appear to be a dirty trick.
Therefore, we add a further example which shows that such extensions can
also appear for valuations of rank one.

\begin{example}                             \label{examp4}
{\bf (for a dependent Artin-Schreier defect extension in rank 1):} \
With $(k(t),v)$ as before, we take $a_1$ to be a root of the
Artin-Schreier polynomial $X^p-X-1/t$. Then $va_1=-1/p<0$. By
induction on $i$, we take $a_{i+1}$ to be a root of the Artin-Schreier
polynomial $X^p-X+a_i\,$, for all $i\in\N$. Then $va_i=-1/p^i<0$.
Note that $t,a_1,\ldots,a_i\in k(a_{i+1})$ for every $i$, because $a_i=
a_{i+1}-a_{i+1}^p$. We have $1/p\in vk(a_1)\setminus vk(t)$. Since
$p\leq (vk(a_1):vk(t))\leq [k(a_1):k(t)]\leq p$, equality holds
everywhere and we find that $vk(a_1)=\frac{1}{p}vk(t)$. Repeating this
argument by induction on $i>1$, we obtain $1/p^i\in vk(a_i)\setminus
vk(a_{i-1})$ and thus, $vk(a_i)= \frac{1}{p} vk(a_{i-1})=\frac{1}{p^i}
vk(t)$. Therefore, the value group of $K:= k(a_i\mid i\in \N)$ is the
$p$-divisible hull $\frac{1}{p^\infty}\Z$ of $\Z$ (an ordered abelian
group of rank $1$).

Finally, we choose $\eta$ such that $\eta^p=1/t$. Since $vK$ is
$p$-divisible and $Kv=k$ is algebraically closed, the extension
$K(\eta)|K$ with the unique extension of the valuation $v$ is immediate.
We wish to determine $\dist(\eta,K)$. We set $c_i:=a_1+ \ldots+ a_{i-1}
\in k(a_{i-1})$ for $i>1$. Using that $a_1^p=\frac{1}{t}+a_1$ and
$a_{i+1}^p=a_{i+1}-a_i$ for $i\in\N$, we compute:
\begin{eqnarray*}
0 & = & \eta^p-\frac{1}{t}\;=\;(\eta-c_i+a_1+\ldots+a_{i-1})^p
-\frac{1}{t}\\
& = & (\eta-c_i)^p+a_1^p+\ldots+a_{i-1}^p-\frac{1}{t}
\;=\; (\eta-c_i)^p+a_{i-1}\;.
\end{eqnarray*}
It follows that $v(\eta-c_i)^p= va_{i-1}\,$, that is, $v(\eta-c_i)=
\frac{1}{p} va_{i-1}=va_i=-1/p^i$. Hence, $-1/p^i\in\Lambda^{\rm L}(\eta,K)$
for all $i$. Assume that there is some $c\in K$ such that $v(\eta-c)>
-1/p^i$ for all $i$. Then $v(c-c_i)=\min\{v(\eta-c_i),v(\eta-c)\}=
-1/p^i$ for all $i$. On the other hand, there is some $i$ such that
$c\in k(a_{i-1})$ and thus, $c-c_i\in k(a_{i-1})$. But this contradicts
the fact that $v(c-c_i)= -1/p^i\notin vk(a_{i-1})$. This proves that the
values $-1/p^i$ are cofinal in $\Lambda^{\rm L}(\eta,K)$. Hence,
$\Lambda^{\rm L}(\eta,K)=vK^{<0}$ and $\dist(\eta,K)=0^-$.

Now Proposition~\ref{ipie} shows that $K$ admits a dependent
Artin-Schreier defect extension. According to this proposition, it
can for instance be generated by a root $\vartheta$ of the polynomial
$X^p-X-(\eta/t)^p$, as $vt=1> \dist(\eta,K)= p\,\dist(\eta,K)$. Then
$\dist(\vartheta,K)=\dist(\eta,K) -1 = 0^- -1=(-1)^-$.

\pars
This example shows that the condition in Lemma~\ref{l5} that the
extension be finite cannot be dropped. Indeed, as we have noted in
Example~\ref{examp1}, $(k(t),v_t)$ is a defectless and hence inseparably
defectless field. But the infinite extension $K$ of $k(t)$ is not an
inseparably defectless field.
\end{example}

\begin{example}                             \label{examp5}
{\bf (for a field having a dependent but no independent Artin-Schreier
defect extension):}
\ We do not know whether the field $K$ of the last example admits any
independent Artin-Schreier defect extension; this an open problem. But in
any case, we can construct from it a field which has a dependent but no
independent Artin-Schreier defect extension. Indeed, by Zorn's Lemma there
is an extension field of $K$ within its algebraic closure not admitting any
independent Artin-Schreier defect extension; such an extension field can be
found by a (possibly transfinitely) repeated extension by independent
Artin-Schreier defect extensions. We choose such an extension field and
call it $L$. Since it is a separable algebraic extension of $K$, the
extension $L(\eta)|L$ is still non-trivial and purely inseparable, and
by our hypothesis on the value group and residue field of $K$, it is
also immediate.

We wish to show that $\dist(\eta,L)\;=\;\dist(\eta,K)$. Assume that
this is not true. Then there is an element $\zeta\in L$ such that
$v(\eta-\zeta)>\dist(\eta,K)$. We write $L=\bigcup_{\mu<\nu} K_\mu$
where $\nu$ is some ordinal, $K_{\mu+1}|K_{\mu}$ is an independent
Artin-Schreier defect extension whenever $0\leq\mu<\nu$, and $K_\lambda=
\bigcup_{\mu<\lambda} K_\mu$ for every limit ordinal $\lambda<\nu$. Let
$\mu_0$ be the minimal ordinal for which $K_{\mu_0}$ contains such an
element $\zeta$. Then $\mu_0$ must be a successor ordinal, and we have
that $\dist(\eta,K)=\dist(\eta,K_{\mu_0-1})$. Hence, $v(\eta-\zeta)>
\dist(\eta,K_{\mu_0-1})$, that is, $\zeta\sim_{K_{\mu_0-1}}\eta$.
But this is a contradiction since by construction, $K_{\mu_0}
|K_{\mu_0-1}$ is an independent Artin-Schreier defect extension. This
proves that
\[\dist(\eta,L)\;=\;\dist(\eta,K)\;=\;0^-\;.\]
Now Corollary~\ref{nodepincomp} shows that $L$ admits a dependent
Artin-Schreier defect extension $L'|L$. On the other hand, by
construction it does not admit any independent Artin-Schreier defect
extension.

This example shows once more that Lemma~\ref{l4} becomes false if the
finiteness condition is dropped. To see this, note that
$L'.L^{1/p^{\infty}}|L^{1/p^{\infty}}$ is still an
Artin-Schreier defect extension, since $L'|L$ is linearly disjoint from
$L^{1/p^{\infty}}|L$, $vL^{1/p^{\infty}}$ is $p$-divisible and
$L^{1/p^{\infty}}v$ is algebraically closed, and the extension of $v$
from $L$ to $L'.L^{1/p^{\infty}}$ and thus also the extension of $v$
from $L^{1/p^{\infty}}$ to $L'.L^{1/p^{\infty}}$ is unique. On the
other hand, $L^{1/p^{\infty}}$ admits no purely inseparable extensions
at all, so by Corollary~\ref{noinsnodep}, such an Artin-Schreier defect
extension can only be independent. We have thus shown that
$L^{1/p^{\infty}}$ admits an independent Artin-Schreier defect extension
whereas $L$ does not. In view of Lemma~\ref{l4}, this is only possible
since $L^{1/p^{\infty}}|L$ is an infinite extension. In contrast to
Example~\ref{examp1}, here we have the case where the lower field is
{\it not} defectless.
\end{example}

\begin{example}                             \label{examp6}
{\bf (for a field which is not relatively algebraically closed in any
maximal immediate extension, but has no independent Artin-Schreier
defect extension):} \
If we replace $k(t)$ by its absolute ramification field $k(t)^r$ (with
respect to an arbitray extension of $v$ to the separable-algebraic
closure of $k(t)$), then the constructions of Example~\ref{examp4}
and~\ref{examp5} can be taken over literally. Since $vk(t)^r$ is
divisible by every prime different from $p$, the value groups of $K$,
$L$ and $L'$ will then be divisible. Since their residue fields are
algebraically closed and all fields are henselian, it follows that $K$,
$L$ and $L'$ are equal to their ramification fields.

Observe that now $L'$ will be contained in every maximal immediate
extension of $L$. This is true because $vL$ is divisible and $Lv$ is
algebraically closed, which implies that every maximal immediate
extension of $L$ is algebraically closed. We have thus shown that $L$ is
not separable-algebraically closed in any of its maximal immediate
extensions, whereas it doesn't admit independent Artin-Schreier
defect extensions.

Since $L'|L$ is linearly disjoint from $L^c|L$, we may replace $L$ by
its completion $L^c$. By Lemma~\ref{ASdown}, $L^c$ still cannot admit
independent Artin-Schreier defect extensions. As the completion of a
henselian field is again henselian (cf.\ [W], Theorem~32.19) and is an
immediate extension, it follows that the completion of a field which is
equal to its absolute ramification field has the same property. The same
argument as before shows that again, $L'.L^c$ will be contained in every
maximal immediate extension of $L^c$. Hence, $L^c$ is an example of a
complete field, equal to its absolute ramification field, which is not
relatively algebraically closed in any maximal immediate extension, but
has no independent Artin-Schreier defect extension.
\end{example}

%
%
\section{Another characterization of defectless fields}  \label{sectac}
\begin{theorem}                                    \label{:}
Let $(K,v)$ be a separably defectless field of characteristic $p>0$. If
in addition $K^c|K$ is separable, then $(K,v)$ is a defectless field.
\end{theorem}
\begin{proof}
Assume that $K^c|K$ is separable, but that $(K,v)$ is not a defectless
field. We have to show that $(K,v)$ is not separably defectless. Let
$(F|K,v)$ be a finite defect extension of minimal degree of
inseparability. If this extension is separable, then we are done.
Suppose it is not. We wish to deduce a contradiction by constructing a
defect extension of smaller degree of inseparability. Let $E|K$ be the
maximal separable subextension. By assumption, it is defectless, so
the purely inseparable extension $(F|E,v)$ must be a defect
extension. Using the arguments of the proof of Theorem~\ref{extrid}
(with $K$ replaced by $E$), one shows that there exists a subextension
$L|E$ of $F|E$ and an element $\eta\in L^{1/p}\setminus L$ such that the
extension $(L(\eta)|L,v)$ is immediate.

Since a finite extension of a complete field is again complete and since
$L^c$ must contain both $K^c$ and $L$, we find that $L^c=L.K^c$.
Together with the fact that $K^c|K$ is separable, this yields that also
$L^c|L$ is separable (see [L], Chapter X, \$6, Corollary 4). It follows
that $\eta\notin L^c$. By an application of Proposition~\ref{ipie}, we
now obtain an immediate separable extension $(L(\vartheta)|L,v)$.
Altogether, we have constructed a defect extension $(L(\vartheta)|K,v)$
which has smaller degree of inseparability than $(F|K,v)$. This is the
desired contradiction.
\end{proof}

We use this theorem to show:

\begin{theorem}                                   \label{csdd}
Let $K$ be a henselian field of characteristic $p > 0$. Then $K$ is a
separably defectless field if and only if $K^c$ is a defectless field.
\end{theorem}
\begin{proof}
Since $K$ is henselian, the same holds for $K^c$ (cf.\ [W],
Theorem~32.19). By virtue of the preceding Theorem, $K^c$ is a
defectless field if and only if it is a separably defectless field. Thus
it suffices to prove that $K^c$ is a separably defectless field if and
only if $K$ is.

Let $L|K$ be an arbitrary finite separable extension. The henselian
field $K$ is separable-algebraically closed in $K^c$ (cf.\ [W],
Theorem~32.19). Consequently, every finite separable extension
of $K$ is linearly disjoint from $K^c$ over $K$, whence
\begin{equation}                            \label{LLc1}
[L.K^c:K^c] \;=\; [L:K]\;.
\end{equation}
On the other hand, $L.K^c =L^c$ is the completion of $L$ and thus an
immediate extension of $L$. Consequently,
\begin{eqnarray}
(vL.K^c:vK^c)\cdot [L.K^cv:K^cv] & = & (vL^c:vK^c)\cdot [L^cv:K^cv]
\nonumber\\
& = & (vL:vK)\cdot [Lv:Kv]\;. \label{LLc2}
\end{eqnarray}

Assume that $K^c$ is a separably defectless field. Then $L.K^c|K^c$ is
defectless, i.e., $[L.K^c:K^c]=(vL.K^c:vK^c)\cdot [L.K^cv:K^cv]$. Hence,
$[L:K] = (vL:vK)\cdot [Lv:Kv]$, showing that $L|K$ is defectless. Since
$L|K$ was an arbitrary finite separable extension, we have shown that
$K$ is a separably defectless field.

\pars
Now assume that $K^c$ is not a separably defectless field. Then there
exists a finite Galois extension $L'|K^c$ with non-trivial defect. Take
an irreducible polynomial $f=X^n+c_{n-1}X^{n-1}+\ldots+c_0\in K^c[X]$ of
which $L'$ is the splitting field. For every $\alpha\in vK$ there are
$d_{n-1},\ldots, d_0\in K$ such that $v(c_i-d_i)\geq\alpha$. If $\alpha$
is large enough, then by Theorem~32.20 of [W], the splitting fields of
$f$ and $g=X^n+d_{n-1}X^{n-1}+\ldots+d_0$ over the henselian field $K^c$
are the same. Consequently, if $L$ denotes the splitting field of $g$
over $K$, then $L'=L.K^c=L^c$. We obtain
\begin{eqnarray*}
[L:K] & \geq & [L.K^c:K^c]\>=\>[L':K^c]\\
& > & (vL':vK^c)[L'v:K^cv] \>=\>(vL^c:vK^c)[L^cv:K^cv]\\
& = & (vL:vK)[Lv:Kv]\;.
\end{eqnarray*}
That is, the separable extension $L|K$ is not defectless. Hence, $K$ is
not a separably defectless field.
\end{proof}

%
%
\section{Algebraically and separable-algebraically maximal fields}
%
%
%
\subsection{Algebraically maximal fields}   \label{sectamf}
We will now give a characterization of algebraically maximal fields
which has been presented by F.~Delon [D1]. We need the following fact,
which was proved by Yu.~Ershov in [Er1] by a different method. Note that
the proof in [D1] has gaps since it is not immediately
clear that if $\sum_{i=1}^{n}\alpha_{i,\nu}$ is increasing with $\nu$,
then there is an increasing cofinal subsequence of
$(\alpha_{i,\nu})_\nu$ for some $i$. Ershov solves this problem by
invoking Ramsey theory. We will avoid this by further analyzing the
valuation theoretical situation.

\begin{lemma}                               \label{vfincrss}
Let $(K,v)$ be any valued field with valuation ring ${\cal O}$, and
$f\in K[X]$ a polynomial in one variable.
\n
1) \ If $v\im_K (f)$ has no maximum, then there is a pseudo Cauchy
sequence $(c_\nu)_{\nu<\lambda}$ of algebraic type in $(K,v)$ without
limit in $K$ but admitting a root of $f$ as a limit, and such that
$(vf(c_\nu))_{\nu<\lambda}$ is a strictly increasing cofinal sequence in
$v\im_K (f)$.
\sn
2) \ If $v\im_{\cal O}(f)$ has no maximum, then there is a pseudo Cauchy
sequence $(c_\nu)_{\nu<\lambda}$ of algebraic type in ${\cal O}$ without
limit in $K$ but admitting a root of $f$ as a limit, and such that
$(vf(c_\nu))_{\nu<\lambda}$ is a strictly increasing cofinal sequence in
$v\im_{\cal O}(f)$.
\end{lemma}
\begin{proof}
1): \ We choose a sequence $(c_\nu)_{\nu<\lambda}$ of elements in $K$
such that the values $vf(c_\nu)$ are strictly increasing and cofinal in
$v\im_K (f)$. We write $f(X)=\prod_{i=1}^{n}(X-a_i)$ with
$a_1,\ldots,a_n\in\tilde{K}$ and choose some extension of $v$ to
$\tilde{K}$.

We introduce a symbol $-\infty$ and define $-\infty<\alpha$ for all
$\alpha\in v\tilde{K}$. Now we consider all balls $B_\alpha^\circ(a_i)=
\{a\in\tilde{K}\mid v(a_i-a)>\alpha\}$ with center a root $a_i$ of $f$,
$1\leq i\leq n$, and radius $\alpha$ in the finite set ${\cal D}:=
\{v(a_i-a_j)\mid 1\leq i<j\leq n\}\cup\{-\infty\}$; note that
$B_{-\infty}^\circ(a_i)= \tilde{K}$. These are finitely many balls, with
$\tilde{K}$ one of them, so there is at least one among them with
$\alpha$ maximal in which there lies some cofinal subsequence of
$(c_\nu)_{\nu<\lambda}$. After renaming our elements if necessary, we
may assume that this ball is $B_\alpha^\circ (a_1)$, that the
subsequence is again called $(c_\nu)_{\nu<\lambda}$, and that exactly
$a_1,\ldots,a_m$ ($m\leq n$) are the roots of $f$ which lie in
$B_\alpha^\circ (a_1)$. Then for every $\nu<\lambda$ and $m<i\leq n$, we
have that
\[v(c_\nu-a_i)\;=\;\min\{v(c_\nu-a_1),v(a_1-a_i)\}\;=\;v(a_1-a_i)\;.\]
On the other hand, by the maximality of $\alpha$ we have the following:
if ${\cal D}$ contains elements $>\alpha$ (which is the case if
$B_\alpha^\circ (a_1)$ contains at least two roots of $f$) and if
$\beta$ is the least of these elements, then there is no cofinal
subsequence of $(c_\nu)_{\nu<\lambda}$ which lies in any of the balls
$B_\beta^\circ(a_i)$. This even remains true if we replace
$B_\beta^\circ(a_i)$ by $B_\beta (a_i) =\{a\in\tilde{K}\mid
v(a_i-a)\geq\beta\}$. Indeed, by our choice of $\beta$ we have for
$1\leq i\leq m$ that $B_\beta (a_i)$ contains $a_1,\ldots,a_m$ and thus,
$c\in B_\beta (a_i)$ implies $v(c-a_j)\geq \beta$ for $1\leq j\leq m$.
If in addition $c$ does not lie in any $B_\beta^\circ(a_j)$, then
$v(c-a_j)=\beta$ for $1\leq j\leq m$. Hence if a cofinal subsequence of
$(c_\nu)_{\nu<\lambda}$ would lie in $B_\beta (a_i)$, then the
value
\[vf(c_\nu)\;=\;v\prod_{i=1}^{n}(c_\nu-a_i)\;=\;\sum_{i=1}^{n}
v(c_\nu-a_i)\;=\;m\beta\,+\,\sum_{i=m+1}^{n} v(a_1-a_i)\]
would be fixed for all $c_\nu$ in this subsequence, a contradiction.

After deleting elements from $(c_\nu)_{\nu<\lambda}$, we may thus assume
that $v(c_\nu-a_i)<\beta\leq v(a_1-a_i)$ for all $\nu$ and $1\leq i\leq
m$. It follows that $v(c_\nu-a_i)=\min\{v(c_\nu-a_1), v(a_1-a_i)\}=
v(c_\nu-a_1)$ for all $\nu$ and $1\leq i\leq m$. Now we compute:
\[vf(c_\nu)\;=\;\sum_{i=1}^{n}
v(c_\nu-a_i)\;=\;mv(c_\nu-a_1)\,+\,\sum_{i=m+1}^{n} v(a_1-a_i)\;.\]
If $\mu<\nu<\lambda$, then $vf(c_\mu)<vf(c_\nu)$ and hence we must have
$v(c_\mu-a_1)<v(c_\nu-a_1)$. This shows that $(c_\nu)_{\nu<\lambda}$ is
a pseudo Cauchy sequence with limit $a_1\,$.

Any limit $a\in\tilde{K}$ of this sequence satisfies $v(a-a_1)>
v(c_\nu-a_1)$ and hence also $v(a-a_i)\geq\min\{v(a-a_1),v(a_1-a_i)\}
>v(c_\nu-a_1)$ for $1\leq i\leq m$ and all $\nu$. Thus,
\[vf(a)\;=\;\sum_{i=1}^{n} v(a-a_i)\;>\;
mv(c_\nu-a_1)\,+\,\sum_{i=m+1}^{n} v(a_1-a_i)\;=\;vf(c_\nu)\;.\]
for all $\nu$. This shows that $a$ cannot lie in $K$. Hence,
$(c_\nu)_{\nu<\lambda}$ is a pseudo Cauchy sequence without limit in
$K$, and by construction, it is of algebraic type.

\mn
2): \ We proceed as in 1), but choose the sequence
$(c_\nu)_{\nu<\lambda}$ in ${\cal O}$ such that the values $vf(c_\nu)$
are strictly increasing and cofinal in $v\im_{\cal O} (f)$. We only have
to note in addition that if $a\in K$ would be a limit of the sequence,
than it would also lie in ${\cal O}$.
\end{proof}

\begin{corollary}                           \label{corKO}
Assume that $(K,v)$ is not $K$-extremal with respect to the polynomial
$f(X)\in K[X]$. Then for all $c\in K$ of large enough value, $(K,v)$ is
not ${\cal O}$-extremal with respect to the polynomial $f(c^{-1}X)$.
Hence, if $(K,v)$ is ${\cal O}$-extremal with respect to every
polynomial in one variable, then $(K,v)$ is $K$-extremal with respect to
every polynomial in one variable. The same holds for ``separable
polynomial'' in the place of ``polynomial''.
\end{corollary}
\begin{proof}
Take the pseudo Cauchy sequence $(c_\nu)_{\nu<\lambda}$ as in
Lemma~\ref{vfincrss}. For large enough $\nu_0<\lambda$, the values of
the $c_\nu$ with $\nu_0<\nu<\lambda$ are constant, say, $\alpha$. For
every $c$ of value $\geq -\alpha$, we have that $cc_\nu\in {\cal O}$ for
$\nu_0<\nu<\lambda$. Hence, $(K,v)$ is not ${\cal O}$-extremal with
respect to the polynomial $f(c^{-1}X)$.
\end{proof}

The first part of the following result was proved by Yu.~Ershov in
[Er1]:
\begin{proposition}                         \label{propersh}
A valued field is algebraically maximal if and only if it is
henselian and $K$-extremal with respect to every polynomial in one
variable. The same holds with ``${\cal O}$-extremal'' in the place of
``$K$-extremal''.
\end{proposition}
\begin{proof}
Suppose that $(K,v)$ is henselian, but not algebraically maximal. Then
there is a proper immediate algebraic extension $L|K$. Take $a\in
L\setminus K$. By Theorem~1 of [Ka], there is a pseudo Cauchy sequence
in $K$ without limit in $K$, having $a$ as a limit. Let $f\in K[X]$ be
the minimal polynomial of $a$ over $K$. Since $K$ is henselian, the
extension of $v$ from $K$ to $K(a)$ is unique. Now it follows from
Lemma~\ref{apCs} that $v\im_K (f)$ has no maximal element. That is, $K$
is not $K$-extremal with respect to $f$. Hence by Corollary~\ref{corKO},
$K$ is also not ${\cal O}$-extremal with respect to every polynomial in
one variable.

For the converse, suppose that there is a polynomial $f\in K[X]$ such
that $v\im_K (f)$ or $v\im_{\cal O}(f)$ has no maximal element. Then by
Lemma~\ref{vfincrss}, $(K,v)$ admits a pseudo Cauchy sequence of
algebraic type in $(K,v)$ without limit in $K$. Now Theorem~3 of [Ka]
shows that there is a proper immediate algebraic extension of $(K,v)$,
i.e., $(K,v)$ is not algebraically maximal.
\end{proof}

Theorem~\ref{extram} and its ``$K$-extremal'' version follow from this
proposition once we have proved the following proposition:
\begin{proposition}                         \label{extrhens}
If a valued field is $K$- or ${\cal O}$-extremal with respect to every
separable polynomial in one variable, then it is henselian.
\end{proposition}
\begin{proof}
In view of Corollary~\ref{corKO} we only have to prove the assertion for
``$K$-extremal''. Suppose that the valued field $(K,v)$ with valuation
ring ${\cal O}$ is not henselian. Then there is a
polynomial $f\in {\cal O}[X]$ and an element $b\in {\cal O}$ such that
$vf(b)>2vf'(b)$, but $f$ has no root in $K$. We take $K_0$ to be a
finitely generated subfield of $K$ containing $b$ and all coefficients
of $f$, and $K_1$ to be the relative algebraic closure of $K_0$ in $K$.
Then $f$ has no root in $K_1\,$, which shows that $K_1$ is not
henselian. Since $K_1$ has finite transcendence degree over its prime
field, it has finite rank, which means that $v|_{K_1}$ is a composition
$v|_{K_1}=v_1\circ\ldots\circ v_k$ of valuations $v_i$ with archimedean
value groups. By a repeated application of Theorem~32.15 of [W], it
follows that $(K_1,v_1)$ is not henselian or for some $i\leq k$ and
$v^i:= v_1\circ \ldots \circ v_{i-1}$, $(K_1v^i,v_i)$ is not henselian.
In the first case, there is a monic separable and irreducible polynomial
$g\in K_1[X]$ with $v_1$-integral coefficients and a $v_1$-integral
element $c\in K_1$ such that $v_1g(c)>2v_1g'(c)$, but $g$ does not have
a zero in $K_1\,$. It follows that $vg(c)>2vg'(c)$.

In the second case, there is a monic separable and irreducible
polynomial $\ovl{g}\in K_1 v^i[X]$ with $v_i$-integral coefficients and
a $v_i$-integral element $\ovl{c}\in K_1 v^i$ such that $v_i\ovl{g}
(\ovl{c})>2v_i\ovl{g}' (\ovl{c})$, but $\ovl{g}$ does not have a zero in
$K_1v^i$. We take some monic polynomial $g\in K_1[X]$ with
$v^i$-integral coefficients such that its $v^i$-reduction is equal to
$\ovl{g}$. Also, we pick a $v^i$-integral element $c\in K_1$ whose
$v^i$-reduction is $\ovl{c}$. Then it follows that $v^{i+1}g(c)>2
v^{i+1}g'(c)$, whence $vg(c)>2vg'(c)$.

It is well known that if $w$ is any valuation for which the polynomial
$g$ has $w$-integral coefficients and $wg(c)>2wg'(c)$ holds, then a
repeated application of the Newton algorithm
\[c_{n+1}\;:=\; c_n-\frac{g(c_n)}{g'(c_n)}\;,\]
starting with $c_0=c$, leads to a strictly increasing sequence of values
$wg(c_n)$; this sequence is cofinal in the value group of $w$ in case
this value group is archimedean. Hence in the first case, we obtain a
sequence of elements $c_n\in K_1$ such that the sequence $v_1
g(c_n)$ is cofinal in $v_1K_1\,$. This implies that if $d\in K$ is such
that $vg(d)$ is the maximum of $v\im_K(g)$ and $H$ denotes the convex
subgroup of $vK$ generated by $v_1K_1$, then $vg(d)>H$. Let $v_H$ be the
coarsening of $v$ with respect to $H$. Then $v_H g(d)>0$, i.e.,
$g(d)v_H=0$. On the other hand, the reduction modulo $v_H$ induces an
isomorphism on $K_1\,$, and since $g$ was chosen to be separable and
irreducible, we thus have that $g'(d)v_H \ne 0$, i.e., $v_Hg'(d)=0$. But
then by the Newton algorithm, if $g(d)\ne 0$, then there is some $d'\in
K$ such that $v_Hg(d')>v_H g(d)$ and hence, $vg(d')>vg(d)$. This
contradiction shows that $g(d)=0$. But this contradicts our choice of
$g$. Hence, $v\im_K(g)$ does not have a maximum.

In the second case, the Newton algorithm provides elements $\ovl{c}_n$
such that the sequence $v_i\ovl{g}(\ovl{c}_n)$ is cofinal in $v_i
(K_1 v^i)\,$. We choose $v^i$-integral elements $c_n\in K_1$ whose
$v^i$-reductions are $\ovl{c}_n\,$. Then it follows that the values
$vg(c_n)$ are cofinal in a convex subgroup $H$ of $vK$ which is the
convex hull of the convex subgroup of $vK_1$ which corresponds to the
coarsening $v^i$ of $v|_{K_1}\,$. This implies that if $d\in K$ is such
that $vg(d)$ is the maximum of $v\im_K(g)$, then $vg(d)>H$. Let $v_H$ be
the coarsening of $v$ with respect to $H$. Then again, $v_H g(d)>0$ and
$g(d)v_H=0$. On the other hand, the reduction of $g$ modulo $v_H$ is
$\ovl{g}$, so $0=g(d)v_H=\ovl{g}(dv_H)$. Since $\ovl{g}$ was chosen to
be separable and irreducible, we thus have that $g'(d)v_H=\ovl{g}'(dv_H)
\ne 0$, i.e., $v_Hg'(d)=0$. Arguing as in the first case, we show that
$v\im_K(g)$ does not have a maximum. Hence we find that $K$ is not
$K$-extremal with respect to every separable polynomial in one variable.
\end{proof}

The following are corollaries to Theorem~\ref{extram}:
\begin{corollary}                           \label{extr1el}
The property ``algebraically maximal'' is elementary in the language of
valued fields.
\end{corollary}

\begin{corollary}                           \label{extr1}
Every extremal field is algebraically maximal.
\end{corollary}

We will now give the
\sn
{\bf Proof of Theorem~\ref{extramh}} and its ``$K$-extremal'' version:
\n
In view of Theorem~\ref{extram}, it suffices to prove that if $K$ is a
henselian but not algebraically maximal field, then there is a
$p$-polynomial $f$ in one variable with coefficients in $K$ with respect
to which $K$ is not $K$-extremal. By Corollary~\ref{corKO}, for suitable
$c\in K$, $K$ is then also not ${\cal O}$-extremal with respect to the
$p$-polynomial $f(c^{-1}X)$.

Take a proper immediate algebraic
extension of $K$. Since $K$ is assumed henselian, it follows
that this extension is purely wild and hence linearly disjoint over
$K$ from the absolute ramification field $K^r$ of $K$. We may assume
that this extension is minimal, that is, it does not admit any proper
subextension. Then by Theorem~13 of [Ku3], it is generated by a root of
a $p$-polynomial $f$. As in the first part of the proof of
Proposition~\ref{propersh} it follows that $v\im_K (f)$ has no
maximal element, that is, $K$ is not $K$-extremal with respect to the
$p$-polynomial $f$. By Corollary~\ref{corKO}, for suitable $c\in K$, $K$
is not ${\cal O}$-extremal with respect to the $p$-polynomial
$f(c^{-1}X)$.                                            \QED

%
%
\subsection{Separable-algebraically maximal fields}  \label{sectsamf}
The following is a further consequence of Proposition~\ref{ipie}:
\begin{corollary}                           \label{samimincomp}
Take a separable-algebraically maximal field $(K,v)$. Every immediate
algebraic extension of $(K,v)$ is purely inseparable and lies in its
completion. Every pseudo Cauchy sequence of algebraic type in $(K,v)$
without limit in $K$ has breadth $\{0\}$, and its unique limit in
$\tilde{K}$ is purely inseparable over $K$.
\end{corollary}
\begin{proof}
Every immediate algebraic extension of $K$ must be purely inseparable
since otherwise, it would contain a proper immediate separable-algebraic
subextension. Since $K$ in particular does not admit any dependent
Artin-Schreier defect extensions, we thus obtain from
Corollary~\ref{+4.5} that every immediate algebraic extension of $K$
must lie in $K^c$.

Take a pseudo Cauchy sequence $(c_\nu)_{\nu<\lambda}$ of algebraic type
in $(K,v)$ without limit in $K$. By Theorem~3 of [Ka], this pseudo
Cauchy sequence gives rise to a proper immediate algebraic extension of
$K$, in which it has a limit. By what we have just shown, this extension
is purely inseparable and lies in the completion of $K$. The latter
shows that $(c_\nu)_{\nu<\lambda}$ has breadth $\{0\}$ and therefore has
a unique limit in the algebraic closure of $K$. The former shows that
this limit must be purely inseparable over $K$.
\end{proof}

The following result has been presented by F.~Delon in [D1]:
\begin{corollary}
The completion of a separable-algebraically maximal field is
algebraically maximal.
\end{corollary}
\begin{proof}
Take any valued field $(K,v)$ and suppose that $K^c$ admits a proper
immediate algebraic extension. Then by Corollary~\ref{expCsalg} there is
a pseudo Cauchy sequence $(c_\nu)_{\nu<\lambda}$ of algebraic type in
$K^c$ without limit in $K^c$. This must have non-trivial breadth, that
is, there is some $\gamma\in vK$ such that $v(c_{\nu+1}-c_\nu)<\gamma$
for all $\nu$ (because otherwise, Theorem~3 of [Ka] would render a
proper immediate extension of $K^c$ within $K^c$, which is absurd).
Since $c_\nu\in K^c$, there is $c^*_\nu\in K$ such that
$v(c_\nu-c^*_\nu) \geq\gamma$ and hence $v(c^*_{\nu+1}-c^*_\nu)=
v(c_{\nu+1}-c_\nu)$ for all $\nu$. It follows that
$(c^*_\nu)_{\nu<\lambda}$ is a pseudo Cauchy sequence in $K$ without
limit in $K$ and with the same non-trivial breadth as
$(c_\nu)_{\nu<\lambda}$.

Let $f\in K^c[X]$ be a polynomial such that for some $\mu<\lambda$, the
sequence $(vf(c_\nu))_{\mu<\nu<\lambda}$ is strictly increasing. Such a
polynomial must exist since by assumption, $(c_\nu)_{\nu<\lambda}$ is of
algebraic type. Since $(c_\nu)_{\nu<\lambda}$ has non-trivial breadth,
it follows from Lemma~8 of [Ka] that the sequence
$(vf(c_\nu))_{\mu<\nu<\lambda}$ is bounded from above in $vK$. Hence,
we can choose a polynomial $f^*\in K[X]$ with coefficients so close to
the corresponding coefficients of $f$ that $vf^*(c_\nu)=vf(c_\nu)$
whenever $\mu<\nu<\lambda$. This shows that also
$(c^*_\nu)_{\nu<\lambda}$ is of algebraic type. Hence by the foregoing
corollary, $K$ cannot be separable-algebraically maximal.
\end{proof}

Now we give the
\sn
{\bf Proof of Theorem~\ref{extrsam}} and its ``$K$-extremal'' version:
\n
Assume that $(K,v)$ is $K$-extremal or ${\cal O}$-extremal with respect
to every separable
polynomial in one variable. Then by Proposition~\ref{extrhens}, $K$ is
henselian. Suppose that $(K,v)$ is not separable-algebraically maximal.
Then there is a proper immediate separable-algebraic extension $L|K$.
Take $a\in L\setminus K$, and let $f\in K[X]$ be the minimal polynomial
of $a$ over $K$. By Theorem~1 of [Ka], there is a pseudo Cauchy sequence
in $K$ without limit in $K$, having $a$ as a limit. Since $K$ is
henselian, the extension of $v$ from $K$ to $K(a)$ is unique. Now it
follows from Lemma~\ref{apCs} that $v\im_K (f)$ has no maximal element,
that is, $K$ is not $K$-extremal with respect to $f$. By
Corollary~\ref{corKO}, it follows that $K$ is not ${\cal O}$-extremal
with respect to the separable polynomial $f(c^{-1}X)$ for some $c\in K$.
This contradicts our assumption that $K$ is $K$-extremal or ${\cal
O}$-extremal with respect to every separable polynomial in one variable.
Hence, $K$ is separable-algebraically maximal.

\pars
For the converse, assume that $(K,v)$ is separable-algebraically
maximal. Suppose that there is a separable polynomial $f\in K[X]$ such
that $v\im_K (f)$ or $v\im_{\cal O} (f)$ has no maximal element. Then by
Lemma~\ref{vfincrss}, $(K,v)$ admits a pseudo Cauchy sequence
$(c_\nu)_{\nu<\lambda}$ of algebraic type in $(K,v)$ without limit in
$K$, but with a root $a\notin K$ of $f$ as a limit. By
Corollary~\ref{samimincomp}, $a$ is purely inseparable over $K$. But
this contradicts the fact that $a$ is a root of a separable polynomial
over $K$. Hence, $K$ is $K$-extremal and ${\cal O}$-extremal with
respect to every separable polynomial in one variable.     \QED

\begin{corollary}                           \label{extrsamel}
The property ``separable-algebraically maximal'' is elementary in the
language of valued fields.
\end{corollary}

We turn to the
\sn
{\bf Proof of Theorem~\ref{extrsamh}:} \
The proof is the same as for Theorem~\ref{extramh}, except that the
immediate algebraic extension of $K$ can be taken to be separable, and
hence the $p$-polynomial $f$ is separable.                         \QED

\parm
Finally, we note that {\bf Theorem~\ref{KO}} has also been proved, as
our above proof of Theorems~\ref{extram}, \ref{extramh}, \ref{extrsam}
and~\ref{extrsamh} have all dealt simultaneously with both $K$- and
${\cal O}$-extremality.

\newcommand{\lit}[1]{\bibitem #1{#1}}

\end{document}